\documentclass{gtpart}     
\usepackage{amsfonts}
\usepackage{latexsym}
\usepackage{mathrsfs}
\usepackage{array}
\usepackage{amscd}
\usepackage[mathcal]{euscript}
\usepackage{epsf, epic, eepic} 
\usepackage{pstricks}

\title{The homotopy theory of  Khovanov homology}

\author{Brent Everitt}
\givenname{Brent}
\surname{Everitt}
\address{Department of Mathematics\\
University of York\\
York\\
YO10 5DD\\
United Kingdom.}
\email{brent.everitt@york.ac.uk}
\urladdr{}

\author{Paul Turner}
\givenname{Paul}
\surname{Turner}
\address{Section de math\'ematiques\\
Universit\'e de Gen\`eve\\
2-4 rue du Li\`evre\\
CH-1211\\
Geneva\\
Switzerland.}
\email{prt.maths@gmail.com}
\urladdr{}
\volumenumber{}
\issuenumber{}
\publicationyear{}
\papernumber{}
\startpage{}
\endpage{}
\doi{}
\MR{}
\Zbl{}
\received{}
\revised{}
\accepted{}
\published{}
\publishedonline{}
\proposed{}
\seconded{}
\corresponding{}
\editor{}
\version{}

\newtheorem{theorem}{Theorem}[section]    
\newtheorem{proposition}{Proposition}[section]    
\newtheorem{corollary}{Corollary}[section]    
\newtheorem{lemma}[theorem]{Lemma}          

\newtheorem*{firstdevice}{presheaf computational tool.}
\newtheorem*{seconddevice}{homotopy computational tool.}

\newtheorem*{firstresult}{Theorem \ref{thm:invlim}.}
\newtheorem*{secondresult}{Proposition \ref{prop:homotopygroups}.}

\theoremstyle{definition}
\newtheorem{definition}[theorem]{Definition}    
\newtheorem*{remark}{Remark}             
\input{boxedeps} 
\SetEPSFDirectory{} 
\HideDisplacementBoxes
\SetRokickiEPSFSpecial  
\input xy 
\xyoption{all} 
\newgray{lightergray}{0.85}
\newcommand{\bN}{\mathbb{N}} 
\newcommand{\bZ}{\mathbb{Z}}

\newcommand{\bB}{\mathbf{B}} 
\DeclareMathOperator{\Map}{Map} 
\DeclareMathOperator{\Hom}{Hom}

\DeclareMathOperator{\holim}{holim}  
\DeclareMathOperator{\hofibre}{hofibre}  
\DeclareMathOperator{\hof}{hof}  

\newcommand{\ra}{\rightarrow} 
\newcommand{\op}{\text{op}}

\newcommand{\ob}{\overline}

\newcommand{\bkh}[1]{\ob{K\!H}^{{#1}}}

\newcommand{\BC}{{\bf C}} 
\newcommand{\BD}{{\bf D}}

\newcommand{\BQ}{{\bf Q}} 
\newcommand{\BQop}{\BQ^\op}

\newcommand{\presh}[1] {\text{\bf PreSh}({#1})}
\newcommand{\preshc} {\presh {\BC}}

\newcommand{\preshq}{\presh {\BQ}}
\newcommand{\invlim}{\varprojlim}

\newcommand{\spaces}{\text{\bf Sp}}
\newcommand{\ptspaces}{\text{\bf Sp}}

\newcommand{\diagc} {\spaces^{\BC}}
\newcommand{\ptdiag}[1] {\spaces^{#1}}
\newcommand{\ptdiagc} {\ptdiag{\BC}}

\newcommand{\XX}[1]{{\bf X}_* {#1}}

\newcommand{\YY}[1]{{\bf Y}_* {#1}}
\newcommand{\YYn}[1]{{\bf Y}_n {#1}}

\newcommand{\diagX}{\mathscr X}
\newcommand{\diagY}{\mathscr Y}
\newcommand{\diagZ}{\mathscr Z}

\newcommand{\scrF}{\mathscr F}
\newcommand{\scrG}{\mathscr G}
\newcommand{\scrH}{\mathscr H}

\DeclareMathAlphabet{\ams}{U}{msb}{m}{n}
\DeclareMathAlphabet{\goth}{U}{euf}{m}{n}
\def\coker{\text{coker}\,}

\def\ker{\text{ker}\,}

\def\ov{\overline}

\def\aa{\alpha}

\def\ss{\sigma}

\def\ve{\varepsilon}

\def\Z{\ams{Z}}

\def\0{\mathbf{0}}
\def\1{\mathbf{1}}
\def\quo{/\kern -.45em\sim}
\newpsobject{showgrid}{psgrid}{subgriddiv=1,griddots=10,gridlabels=6pt,gridcolor=red}
\def\blob{\bullet}
\def\Langle{\langle\kern -2pt\langle}
\def\Rangle{\rangle\kern -1.9pt\rangle}
\newcommand{\ab}{\mathbf{Ab}}
\newcommand{\catC}{\mathbf{C}}
\newcommand{\catCop}{\mathbf{C}^{\text{op}}}

\newcommand{\BP}{\mathbf{P}}

\newcommand{\rmod}{\vrule width 0mm height 0 mm depth 0mm_R\mathbf{Mod}}
\newcommand{\Zmod}{\vrule width 0mm height 0 mm depth
  0mm_\Z\mathbf{Mod}}
\newcommand{\prsh}{\mathbf{PreSh}}
\begin{document}

\begin{abstract}    
We show that the unnormalised Khovanov homology of an oriented link
can be identified 
with the derived functors of the inverse limit. This leads to a
homotopy theoretic interpretation of Khovanov homology.
\end{abstract}

\maketitle


\section*{Motivation and introduction}

In order to apply the methods of homotopy theory to Khovanov homology
there are several natural approaches. One is to build a space or
spectrum whose classical invariants give Khovanov homology, then show
its homotopy type is a link invariant, and finally study this space using homotopy
theory. Ideally this approach would begin with some interesting
geometry and lead naturally to Khovanov homology.  One also might hope
to construct something more refined than Khovanov homology in this
way (see Lipshitz-Sarkar \cite{LipshitzSarkar} for a combinatorial approach to
this). 
Another approach is to interpret the
existing constructions of Khovanov homology in homotopy theoretic
terms. 
By placing the
constructions into a homotopy setting one makes Khovanov homology
amenable to the methods and techniques of homotopy theory. In this paper our interest is
with the second of these approaches. Our aim is to show that Khovanov
homology 
can be interpreted in a homotopy theoretic way using homotopy limits
and to 
subsequently develop a number of results about the specific type of
homotopy limit arising. 
The latter will provide homotopy tools appropriate for studying Khovanov homology.

Recall that the central combinatorial input for Khovanov homology
is the decorated ``cube'' of resolutions based on
a link diagram $D$ (see Section \ref{subsection:boolean}). As we
explain later, it is  convenient
to view this cube as a presheaf of abelian groups over a certain poset $\BQ$,
that is, as a functor  $F_{KH} \colon \BQop \ra \ab $.

In the first section we show that Khovanov homology can
be described in terms of the right derived functors of the inverse
limit of this presheaf. 

\begin{firstresult}
Let $D$ be a link diagram and let $F_{KH} \colon\BQop\ra\ab$ be
the Khovanov presheaf defined in \S\ref{subsection:boolean}. Then, 
$$
 \ov{KH}^i ({D}) \cong {\invlim_{\BQop}}^i {F_{KH}}
$$
\end{firstresult}

On the left we have singly graded unnormalised Khovanov homology (see
Section \ref{subsection:boolean}) 
while on the right we have the $i$-th derived functor of the inverse
limit (see Section \ref{subsection:derived}).
This result is central to the homotopy theoretic interpretation of
Khovanov homology but is also of independent interest: many cohomology
theories are defined as the right derived functors of some interesting
partially exact functor, or at least can be described in such
terms. Examples include group cohomology, sheaf cohomology and
Hochschild cohomology. Obtaining a description in these terms for
Khovanov homology reveals its similarity to existing theories not 
apparent from the original definition. Moreover it opens up
Khovanov homology to the many techniques available to cohomology
theories defined as right derived functors. Also the construction
given in this paper is functorial with respect to morphisms of
presheaves, which being more general, may offer calculational
advantage. By connecting with  a more familiar description of higher 
derived functors we also obtain a description
of Khovanov homology as  the cohomology of the
classifying space equipped with a system of local coefficients as 
described in Proposition \ref{prop:class}.

Right derived functors of a presheaf of abelian groups can be
interpreted in homotopy theoretic terms by way of the homotopy limit
of the corresponding diagram of Eilenberg-Mac Lane spaces.
In the second section we 
recall basic facts about homotopy limits before
returning to  Khovanov homology. We compose the
Eilenberg-Mac Lane space functor $K(-,n)$ with the Khovanov 
presheaf $F_{KH}$ of a link diagram to obtain a
diagram of spaces $\scrF_n\colon \BQop \ra \ptspaces $ 
whose homotopy limit $\YYn D = \holim_{\BQop}\scrF_n$ has 
homotopy groups described in the following proposition.

\begin{secondresult}
$$
  \pi_i(\YYn {D}) \cong 
\begin{cases}
\bkh {n-i} ({D}) & 0\leq i \leq n,\\
0 & \text{ else.}
\end{cases}
$$
\end{secondresult}

For rather elementary reasons the space $\YYn D $ is seen to be a product of
Eilenberg-Mac Lane spaces and thus determined by the Khovanov
homology. Thus the problem of  defining an invariant space or spectrum (a {\em homotopy type})
is ``solved'' by the above as well, but
in an uninteresting way.  Nevertheless we now find ourselves within a homotopy theory
context so can apply its methods and techniques to Khovanov homology.

In the third section we develop this perspective further by isolating
a result about holim and homotopy fibres in this specific situation
which may be useful in the study of Khovanov homology. One central
point is that in the presheaf setting (or using chain complexes) one
has long exact sequences in homology arising from short exact
sequences of presheaves. 
Typically the latter arise from a given injection or surjection and
one requires some luck for this to be the case. In the homotopy
setting, by contrast, {\em any} map of spaces has a homotopy fibre and
an attendant long exact sequence in homotopy groups. 
We illustrate the use of this calculus in the last section where we
discuss the skein relation as the homotopy long exact sequence of the
crossing change map, reprove Reidemeister invariance from the homotopy
perspective and make an explicit computation.

We have tried as far as possible to make this article readable  both
by knot theorists interested in Khovanov homology and by homotopy
theorists with a passing 
interest in knot theory. 

\subsection*{Acknowledgements}

We thank Hans-Werner Henn, Kathryn Hess, Robert Lipshitz,  Sucharit Sarkar,
J\'er\^ome Scherer and the referee for helpful remarks.

\section{Khovanov homology and higher inverse limits}
\label{section:khovanov}

The main result of this section is a
reinterpretation of the (unnormalised) Khovanov homology of a link
as the derived functors of $\invlim$ over a certain small category.

\subsection{A modified Boolean lattice and the inverse limit}
\label{subsection:boolean}

Let $\bB=\bB_A$ be the Boolean lattice on a set $A$: the
poset of subsets of $A$ ordered by \emph{reverse\/}
inclusion. We write $\leq$ for the partial order and $\prec$ for the
covering relation, i.e.: $x\leq y$ when subsets $x\supseteq y$ and  
$x\prec y$ when $x$ is obtained from $y$ by adding a single element.

Now let ${D}$ be a link diagram and $\bB$ the Boolean lattice on the set of
crossings of ${D}$.
Each crossing of ${D}$ can be $0$- or $1$-resolved
$$
\begin{pspicture}(0,0)(13,0.8)
%
\rput(-0.5,-0.1){
\rput(7,0){
\psline(0.2,0.75)(-0.2,0.25)
\psframe[fillstyle=solid,fillcolor=white,linecolor=white](-0.075,0.425)(0.075,0.575)
\psline(-0.2,0.75)(0.2,0.25)
\rput(-2,0){
\psline{<-}(0.5,0.5)(1.5,0.5)
\rput(1,0.75){${\scriptstyle 0}$}
}
\rput(0,0){
\psline{->}(0.5,0.5)(1.5,0.5)
\rput(1,0.75){${\scriptstyle 1}$}
}
\rput(-2,0){
\psbezier[showpoints=false](-0.25,0.7)(-0.2,0.65)(-0.15,0.6)(0,0.6)
\psbezier[showpoints=false](0.25,0.7)(0.2,0.65)(0.15,0.6)(0,0.6)
\psbezier[showpoints=false](-0.25,0.3)(-0.2,0.35)(-0.15,0.4)(0,0.4)
\psbezier[showpoints=false](0.25,0.3)(0.2,0.35)(0.15,0.4)(0,0.4)
}
\rput(2,0){
\psbezier[showpoints=false](-0.2,0.75)(-0.15,0.7)(-0.1,0.65)(-0.1,0.5)
\psbezier[showpoints=false](-0.2,0.25)(-0.15,0.3)(-0.1,0.35)(-0.1,0.5)
\psbezier[showpoints=false](0.2,0.75)(0.15,0.7)(0.1,0.65)(0.1,0.5)
\psbezier[showpoints=false](0.2,0.25)(0.15,0.3)(0.1,0.35)(0.1,0.5)
}
}
}
\end{pspicture}
$$
and if $x$ is some subset of
the crossings, then the complete resolution ${D}(x)$ is what results from $1$-resolving
the crossings in $x$ and $0$-resolving the crossings not in $x$. It is
a collection of planar circles.

Let $V=\Z[1,u]$ where $\Z[S]$ is the free abelian group on the set
$S$. This rank two abelian group becomes a Frobenius algebra using the maps
 $m\colon V\otimes V\rightarrow V$, $\epsilon\colon V\rightarrow \Z$ and
$\Delta\colon V\rightarrow V\otimes V$ defined by
\begin{align*}
m&:1\otimes 1\mapsto 1,\,\,\,\,\,\,1\otimes u\text{ and }u\otimes 1\mapsto u,\,\,\,\,\,\,
u\otimes u\mapsto 0  \\
\epsilon &:1\mapsto 0,\,\,\,\,\,\,
u\mapsto 1\\
\Delta&:1\mapsto 1\otimes u+u\otimes 1,\,\,\,\,\,\,
u\mapsto u\otimes u.
\end{align*}
The ``Khovanov cube'' is obtained by
assigning 
abelian groups to the elements of $\bB$ and homomorphisms
between the groups associated to comparable elements.
One says ``cube'' as the Hasse diagram of the poset $\bB_A$
is the $|A|$-dimensional cube, with edges given by the covering relations. 

For $x\in\bB$ let $F_{KH}(x)=V^{\otimes k}$, with
a tensor factor corresponding to each connected component of ${D}(x)$.
If $x\prec y$ in $\bB$ then ${D}(x)$
results from $1$-resolving a crossing that was $0$-resolved in ${D}(y)$,
with the qualitative effect that two of the circles in ${D}(y)$
fuse into one in ${D}(x)$, or one of the circles in ${D}(y)$ bifurcates
into two in ${D}(x)$.  In the first case 
$F_{KH}(x\prec y):F_{KH}(y)=V^{\otimes k}\rightarrow V^{\otimes k-1}=F_{KH}(x)$ 
is the map
using $m$ on the tensor factors corresponding to the fused circles,
and the identity on the others.  In the second, 
$F_{KH}(x\prec y):F_{KH}(y)=V^{\otimes k}\rightarrow V^{\otimes
k+1}=F_{KH}(x)$ 
is the map using $\Delta$ on the tensor factor
corresponding to the bifurcating circles, and the identity on the
others.

All of this is most concisely expressed by regarding $\bB$ as a category with
objects the elements of $\bB$ and with a unique morphism $x\ra y$
whenever $x\leq y$. The decoration by abelian groups is then
nothing other than a covariant functor, or presheaf,
$$
F_{KH} \colon \bB^\op \ra \ab
$$
where $\ab$ is the category of
abelian groups. 
The diagram ${D}$ is suppressed from the notation.

Each square face of the cube $\bB$ is sent by the functor $F_{KH}$ to
a commutative diagram of abelian groups. To extract a cochain complex
from the decorated cube these squares must
\emph{anti\/}commute, and this is achieved by adding $\pm$ signs to the
edges of the cube so that each square face has an odd number of $-$
signs on its edges. 
We write $[x,y]$ for the sign associated to the edge
$x\prec y$ of $\bB$. 
The Khovanov complex $K^*$ has $n$-cochains
$K^n=\bigoplus_{|x|=n} F_{KH}(x)$
the direct sum over the subsets of size $n$ (or rows of the cube), and differential
$d:K^{n-1}\rightarrow K^n$ given by $d=\sum [x,y]F_{KH}(x\prec y)$,
the sum over all pairs $x\prec y$ with $x$ of size $n$ (or sum of
all signed maps between rows $n-1$ and $n$). That $d$ is
a differential follows immediately from the anti-commuting of the
signage.

\begin{definition}\label{def:kh}
 The unnormalised Khovanov homology of a link diagram
${D}$ is defined as the homology of the Khovanov cochain complex:
$$
\ov{KH}^*({D})= H(K^*, d).
$$
The normalised Khovanov homology of an oriented link diagram
${D}$ with $c$ negative crossings is a shifted version of the above:
$$
KH^*({D})= \ov{KH}^{*+c}({D})
$$
\end{definition}

The normalised Khovanov homology is a link invariant. 
All of the above is standard and there are several reviews of this
material available (see for example Bar-Natan \cite{Bar-Natan02},
Turner \cite{Turner06} and Khovanov \cite{khovanov06}).

\paragraph{A note on the $q$-grading.} Usually there is an internal
grading on Khovanov homology making it a bigraded theory. This 
``$q$-grading'' is important in recovering the Jones
polynomial. A huge amount of information is retained however even if this grading
is completely ignored. For example Khovanov homology detects the unknot 
with or without the $q$-grading. In this paper the $q$-grading plays 
no role and we consider the Frobenius algebra $V$ above as ungraded, 
resulting in a singly graded theory. 

\paragraph{}For what follows we need to modify the poset $\bB$ in a
seemingly innocuous way, but
one which has considerable consequences (see also the remarks at the
end of Section \ref{subsection:projective}). There is a unique maximal
element $\1\in\bB$ (corresponding to the empty subset of $A$) with
$x\leq\1$ for all $x\in\bB$.
Now formally adjoin to $\bB$ an additional maximal element $\1'$ such that
$x\leq\1'$ for all $x\in\bB$ with $x\not=\1$,
and denote the
resulting poset (category)  by $\BQ=\BQ_A$. 
Extend $F_{KH}$ to a (covariant)
functor
$$
F_{KH} \colon \BQop \ra \ab
$$
by setting $F_{KH}(\1')= 0$ and $F_{KH}(x\ra \1'): F_{KH}(\1') \ra
F_{KH}(x)$ to be the only possible homomorphism. 

The construction of $K^*$ extends verbatim to $\BQ$:
the chains are the direct sum over the rows of $\BQ$ (identical to $\bB$ except
for the top row where the zero group is added) and the
differential is the sum of signed maps between
consecutive rows -- again identical except between the first and second
rows; we adopt the convention $[x,\1']=-1$ for an $x$ with $x\prec\1'$. 
The resulting homology is easily seen to be the unnormalised
Khovanov homology again.

It will be convenient later to identify $\BQ$ with the poset of cells
of a certain CW complex. Recall that a CW complex $X$ is regular 
if for any cell $x$ the characteristic map
$\Phi_x:(B^k,S^{k-1})\rightarrow(X^{k-1}\cup x,X^{k-1})$ is a
homeomorphism of $B^k$ onto its image. We can then define a partial
order on the cells of $X$ by $x\leq y$ exactly when 
$\ov{x}\supseteq y$, 
where $\ov{x}$ is the (CW-)closure of the cell. 

To realise $\BQ_A$ as such a thing suppose that $|A|=n$ and let
$\Delta^{n-1}$ be an $(n-1)$-simplex. Let $X$ be the suspension
$S\Delta^{n-1}$, an $n$-ball, and take the obvious CW decomposition of
$X$ with two $0$-cells (the suspension points) and all other cells the
suspensions $Sx$ of the cells $x$ of $\Delta^{n-1}$. As the suspension
of cells preserves the inclusions $\ov{x}\supseteq y$ and the two
$0$-cells are maximal with respect to this we get $X$ has cell poset
$\BQ$. An $x\in\BQ$ corresponds to an $|x|$-dimensional cell of $X$;
the case $n=3$ is in Figure \ref{fig:orange}.

\begin{figure}
  \centering
\begin{pspicture}(0,0)(13,4.5)
\rput(4.4,1.8){
\rput(0,0){\BoxedEPSF{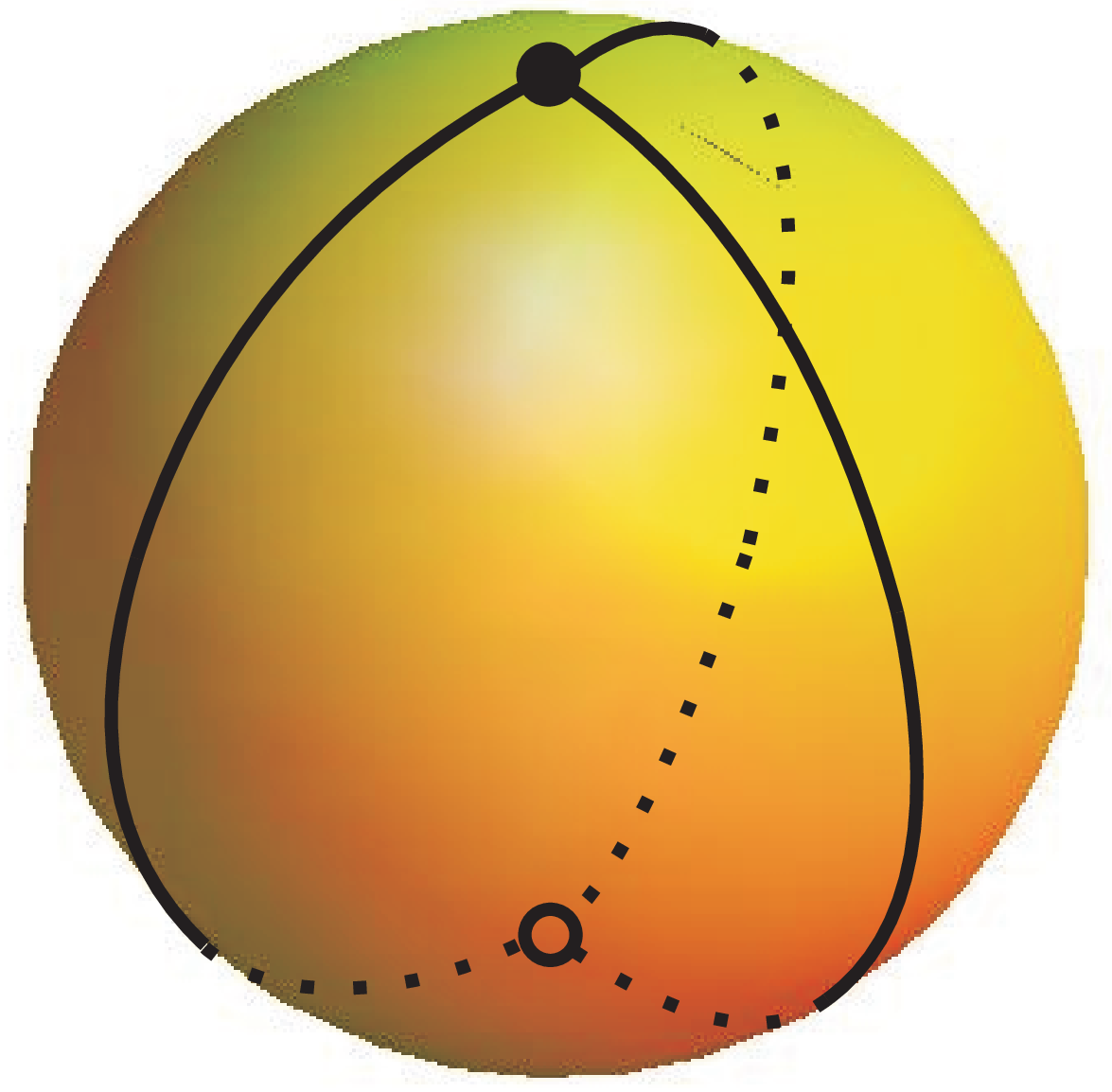 scaled 300}}
\rput(0.1,1.4){$\1$}
\rput(0.15,-0.7){$\1'$}
}
\rput(9,2){
\rput(0,0){\BoxedEPSF{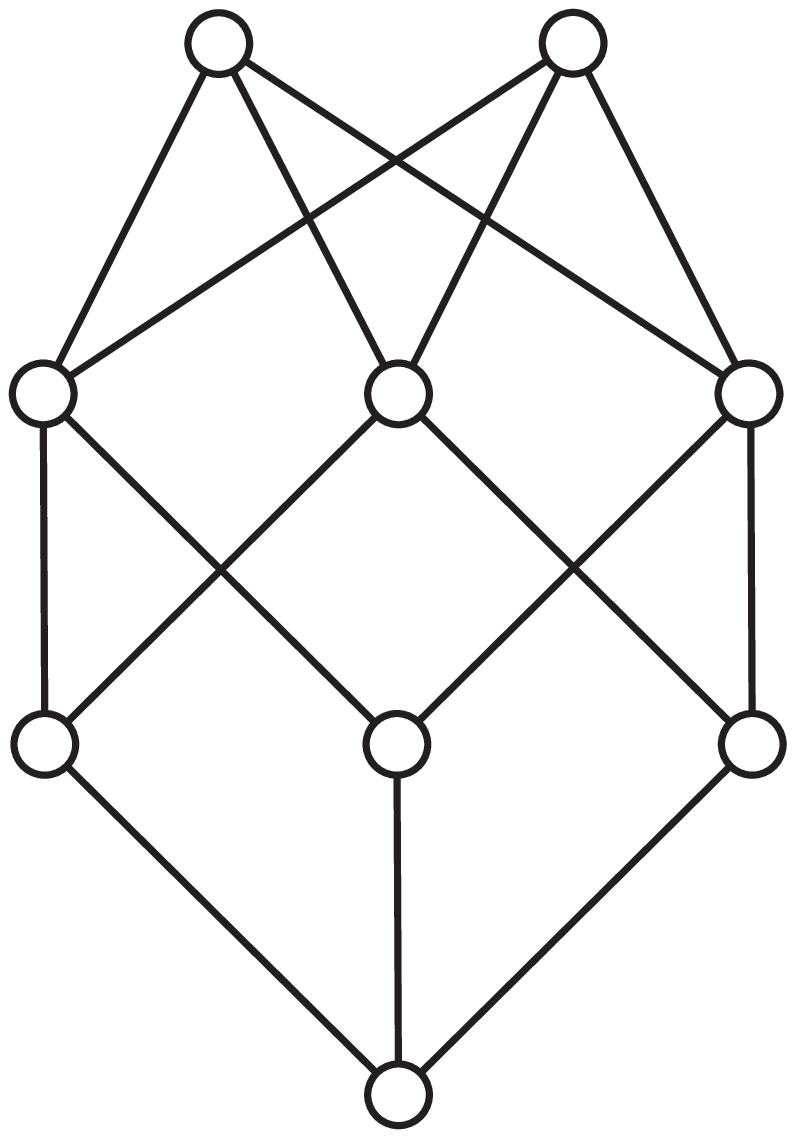 scaled 300}}
\rput(-0.55,1.9){$\1$}
\rput(0.55,1.9){$\1'$}
}
\end{pspicture}
\caption{Regular CW complex $X$ (\emph{left}) with cell poset $\BQ_A$ 
(\emph{right})
  for $|A|=3$.}
  \label{fig:orange}
\end{figure}

Using the signage
introduced above, if $x$ is a $1$-cell 
we have $[x,\1]+[x,\1']=0$; if $\dim x-\dim
y=2$ and $z_1,z_2$ are the unique cells with $x\prec
z_i\prec y$,
then $[x,z_1][z_1,y]+[x,z_2][z_2,y]=0$. These properties then ensure
that there are orientations
for the cells of $X$ so that $[x,y]$ is the incidence number of the
cells $x$ and $y$ (see Massey \cite[Chapter IX, Theorem 7.2]{Massey91}).

We finish this introductory subsection by recalling the definition of the inverse limit
of abelian groups. 
Let $\catC$ be a small category and
$F:\catC\rightarrow\ab$ a functor. Then the
inverse limit $\invlim_{\catC} F$ is an abelian group that is
universal with respect to the property that for all $x\in\catC$ there are
homomorphisms $\invlim_{\catC} F\rightarrow F(x)$
that commute with the homomorphisms $F(x)\rightarrow F(x')$ for all
morphisms $x\rightarrow x'$ in $\catC$. The limit
is constructed by taking the subgroup of the product
$\Pi_{x\in\catC}F(x)$ consisting of those $\catC$-tuples
$(\aa_x)_{x\in\catC}$ such that for all morphisms $x\ra x'$, 
the induced map $F(x)\rightarrow F(x')$ sends
$\aa_x$ to $\aa_{x'}$. 

It is an easy exercise to see that $\invlim_{\BQop} F_{KH}=\ker d^0$,
the degree zero differential of the cochain complex $K^*$, and so
\begin{equation}
  \label{eq:1.1.1}
\invlim_{\BQop} F_{KH} \cong \bkh 0 ({{D}})
\end{equation}

\subsection{Derived functors of the inverse limit}
\label{subsection:derived}

We have seen that presheaves of abelian groups 
provide a convenient language for the construction
of Khovanov homology, and that the inverse limit 
of the presheaf $F_{KH}:\BQop\rightarrow\ab$ 
captures this homology in degree zero. In this subsection we review
general facts about the category of presheaves, the inverse limit functor and
its derived functors. These ``higher limits'' give, by definition, the
cohomology of a small category $\catC$ with coefficients in a
presheaf. The moral is that they are computed using projective
resolutions for the trivial (or constant) presheaf, just as group
cohomology, say, is computed using projective resolutions for the trivial
$G$-module.
The material here is standard (see
e.g. Weibel \cite[Chapter 2]{Weibel94}) and obviously holds in greater
generality; rather than working in the category $\rmod$ of modules
over a commutative ring $R$, we content ourselves with
$\ab:=\Zmod$. 
 In the following subsection we will show that
the higher limits capture Khovanov homology in \emph{all\/} degrees,
not just degree zero. 

Recall that a presheaf on a small category $\catC$ is a (covariant) functor 
$F:\catCop\rightarrow\ab$. The category
$\prsh(\catC)=\ab^{\catCop}$ has objects the presheaves
$F:\catCop\rightarrow\ab$ and morphisms the natural
transformations $\tau:F\rightarrow G$. For $x\in\catC$ we write $F(x)$
for its image in $\ab$ and $\tau_x$ for the map $F(x)\rightarrow G(x)$
making up the component at $x$ of the natural transformation $\tau$.

$\prsh(\catC)$ is an abelian category having enough projective and injective
objects.
Many basic
constructions in $\prsh(\catC)$, such as kernels, cokernels, decisions
about exactness, etc, can be constructed
locally, or ``pointwise'', e.g. the value of the presheaf $\ker(\tau:F\rightarrow
G)$ at $x\in\catC$ is $\ker(\tau_x:F(x)\rightarrow G(x))$, and
similarly for images. In particular, a sequence of presheaves
$F\rightarrow G\rightarrow H$ is exact  if and only if for
all $x\in\catC$ the local
sequence $F(x)\rightarrow G(x)\rightarrow H(x)$ is exact.

The simplest presheaf is the constant one: if $A\in\ab$ 
define $\Delta A:\catCop\rightarrow\ab$ by $\Delta A(x)=A$ for all
$x$ and for all morphisms
$x\rightarrow y$ in $\catC$ let $\Delta A(x\rightarrow
y)=1:\Delta A(y)\rightarrow\Delta A(x)$. 
If $f:A\rightarrow B$ is a map of abelian groups then there is a natural
transformation $\tau:\Delta A\rightarrow\Delta B$ with
$\tau_x:\Delta A(x)\rightarrow\Delta B(x)$ the map $f$. Thus
we have the constant sheaf functor
$\Delta:\ab\rightarrow\prsh(\catC)$
which is easily seen to be exact. 

We saw at the end of \S\ref{subsection:boolean}
that the inverse limit $\invlim F$ exists
in $\ab$ for any presheaf $F\in\prsh(\catC)$. Indeed, we have a (covariant) functor
$\invlim:\prsh(\catC)\rightarrow\ab$
by universality. 
For any $A\in\ab$ and any $F\in\prsh(\catC)$ there are natural bijections
\begin{equation}
  \label{eq:1.2.1}
\Hom_{\prsh(\catC)}(\Delta A,F)\cong\Hom_\Z(A,\invlim F),
\end{equation}
so that $\invlim$ is right adjoint to
$\Delta$. In particular $\invlim$
is left exact, and we have the right derived functors
$$
\textstyle{\invlim^i:=R^i\invlim:\prsh(\catC)\rightarrow\ab}\,\,\,\,\,\,\,(i\geq
0)
$$
with $\invlim^0$ naturally isomorphic to $\invlim$.

A special case of the adjointness
(\ref{eq:1.2.1}) is the following: for any presheaf $F$ over $\catC$ the universality of the limit gives a
homomorphism $\Hom_{\prsh(\catC)}(\Delta\Z,F)\rightarrow\invlim F$
that sends a natural transformation $\tau\in \Hom_{\prsh(\catC)}(\Delta\Z,F)$ to the tuple
$(\tau_x(1))_{x\in\catC}\in\invlim F$. This is in fact a natural
isomorphism, so we have a natural isomorphism of functors
$$
\invlim\cong\Hom_{\prsh(\catC)}(\Delta\Z,-)
$$
and thus
\begin{equation}
  \label{eq:1.2.3}
\textstyle{\invlim^i\cong
  R^i\Hom_{\prsh(\catC)}(\Delta\Z,-)}\,\,\,\,\,\,\,
\text{for all $i\geq 0$.}
\end{equation}
If $0\rightarrow F\rightarrow G\rightarrow H\rightarrow 0$ is a short
exact sequence in $\prsh(\catC)$ then there is a long exact sequence
in $\ab$: 
\begin{equation}
  \label{eq:10}
0\longrightarrow
\invlim F\longrightarrow
\invlim G\longrightarrow
\invlim H\longrightarrow
\cdots
{\textstyle \invlim^i} F\longrightarrow
{\textstyle \invlim^i} G\longrightarrow
{\textstyle \invlim^i} H\longrightarrow
\end{equation}

It turns out that the derived functors of the covariant $\Hom$ functor
in (\ref{eq:1.2.3}) can be replaced by the derived functors of the contravariant
$\Hom$ functor. Let $F,G$ be presheaves over the small category $\catC$. Then
$$
R^i\Hom_{\preshc}(F, - )(G) \cong R^i\Hom_{\preshc}(-,G)(F)
$$
for all $i\geq 0$. One thinks of this as a ``balancing Ext'' result
for presheaves. The corresponding result in $\rmod$ is \cite[Theorem 2.7.6]{Weibel94}, and the reader can check that the proof given there goes
straight through in $\prsh(\catC)$. Summarizing, for $F\in\prsh(\catC)$
\begin{equation}
  \label{eq:1.2.4}
\textstyle{
\invlim^i(F)
\cong (R^i\Hom_{\prsh(\catC)}(\Delta\Z,-))(F)
\cong (R^i\Hom_{\prsh(\catC)}(-,F))(\Delta\Z)
}.  
\end{equation}

To compute the right derived functors of a contravariant functor
like $\Hom_{\prsh(\catC)}(-,F)$, we use
a projective resolution. 
Let $P_*\rightarrow\Delta\Z$ be a projective resolution for
$\Delta\Z$, i.e.: an exact sequence 
\begin{equation}
  \label{eq:1.2.5}
\cdots 
\stackrel{\delta}{\longrightarrow}
P_2 
\stackrel{\delta}{\longrightarrow}
P_1
\stackrel{\delta}{\longrightarrow}
P_0  
\stackrel {\varepsilon}{\longrightarrow}
\Delta\Z
\longrightarrow
0
\end{equation}
with the $P_i$ projective presheaves. Then the final term in
(\ref{eq:1.2.4}) is the degree $i$ cohomology of the cochain complex
$\Hom_{\prsh(\catC)} (P_*,F)$:
\begin{equation}
  \label{eq:1.2.6}
\cdots
\stackrel{\delta^*}{\longleftarrow}
\Hom_{\prsh(\catC)} (P_1,F)
\stackrel{\delta^*}{\longleftarrow}
\Hom_{\prsh(\catC)} (P_0,F)
\longleftarrow
0
\end{equation}

\subsection{A projective resolution of $\Delta\bZ$ and the Khovanov
  complex}
\label{subsection:projective}

We return now to the particulars of \S\ref{subsection:boolean} and
compute the cochain complex (\ref{eq:1.2.6}) when $F=F_{KH}$, the Khovanov
presheaf in $\prsh(\BQ)$ where $\BQ$ is the poset of \S\ref{subsection:boolean}. 
To do this we present a particular 
projective resolution for the constant presheaf $\Delta\Z$ on $\BQ$.

We start by constructing a presheaf $P_n$ in $\prsh(\BQ)$ for each
integer $n>0$ . 
Remembering that $\BQ$ is the cell poset of the
regular CW complex $X$ of \S\ref{subsection:boolean}, for $x\in\BQ$
set
$$
P_n(x):=\Z[\text{$n$-cells of $X$ contained in the closure of the cell $x$}].
$$
Thus if $\dim x < n$ then $P_n(x)=0$; if $\dim x = n$ then $P_n(x) =
\Z[x]\cong\Z$; 
and if $\dim x >n$ then $P_n(x)$ is a direct sum of copies of $\Z$,
one copy for each $n$-cell in the boundary of $x$.  
If $x\leq y$ in $\BQ$ then we take $P_n(x\leq y)\colon P_n(y) \ra P_n(x)$ to be 
the obvious inclusion.

For a given presheaf $F\in\prsh(\BQ)$ there is a nice characterization of the group of presheaf
morphisms $P_n\rightarrow F$:

\begin{proposition}
\label{subsection:projective:prop1}
For $F\in\prsh(\BQ)$ the map 
$$
f^n: \Hom_{\prsh(\BQ)}(P_n,F)\ra\kern-1mm\bigoplus_{\dim x=n}\kern-1mmF(x)
$$
defined by $f^n(\tau) = \sum_{\dim x=n} \tau_x(x)$, is an isomorphism of abelian groups.
\end{proposition}

\begin{proof}
That $f^n$ is a homomorphism is clear since $(\tau+\ss)_x=\tau_x+\ss_x$.
To show injectivity, suppose that $f^n(\tau) =0$ from which it follows
that $\tau_x(x) = 0$ for all $n$-cells $x\in\BQ$.
To show that $\tau =0$ we must prove that $\tau_y\colon P_n(y) \ra
F(y)$ is zero for all $y\in \BQ$. For $\dim y< n $ there is nothing to
prove since $P_n(y)=0$. For $\dim y = n$ we have $P_n(y) = \Z[y]$
and $\tau_y(y) = 0$ since $y$ is an $n$-cell. For $\dim y > n$ 
$$
P_n(y) = \Z[y_\alpha \mid \dim y_\alpha =n \text{ and } y_\alpha \text{ in the closure of } y]
$$
and we have
$$
\tau_y(y_\alpha) = \tau_y (P_n(y\leq y_\alpha)(y_\alpha)) = F(y\leq y_\alpha)(\tau_{y_\alpha}(y_\alpha)) = 0.
$$
The first equality since $P_n(y\leq y_\alpha)\colon P_n(y_\alpha) \ra
P_n(y)$ is an inclusion, and the second by naturality of $\tau$ and
the third since $y_\alpha$ is an $n$-cell. 
Finally, $f^n$ is surjective because $P_n(x)$ is free and so there
is no restriction on the images $\tau_x(x)\in F(x)$ .
\end{proof}

The isomorphism given in Proposition \ref{subsection:projective:prop1}
allows us to  define a morphism
$\tau:P_n\rightarrow F$ by specifying a tuple $\sum\lambda_x\in\oplus
F(x)$, where the sum is over the $n$-cells $x$.

It is easy to see that the $P_n$ are projective presheaves. 
Given the following diagram of presheaves and morphisms (with solid arrows) and
exact row:
\begin{equation*}
\xymatrix{
&P_n \ar[d]^-{\tau}\ar@{.>}[ld]_-{\widehat{\tau}}& 
\\
G \ar[r]^-{\ss} & F \ar[r] & 0
}  
\end{equation*}
then the local maps $G(x)\stackrel{\ss_x}{\longrightarrow}F(x)$ are
surjections. Thus if $\sum\lambda_x\in\oplus
F(x)$ specifies the map $\tau$
then for each $x$ there is a $\mu_x\in G(x)$ with
$\ss_x(\mu_x)=\lambda_x$. Hence there exists a morphism
$\widehat{\tau}: P_n\rightarrow G$ specified by $\sum\mu_x$, which
clearly makes the diagram commute. The $P_n$
are thus projective presheaves.

We now assemble the $P_n$'s into a resolution of $\Delta\Z $ by
defining maps $\delta_n:P_n\rightarrow P_{n-1}$. For $x\in\BQ$ 
let $\delta_{n,x}\colon P_n(x) \ra P_{n-1}(x)$ be the homomorphism
defined
by
$$
\delta_{n,x}(y)=\sum_{y\prec z} [y,z]\,z
$$
for $y$ an $n$-cell $\subset \ov{x}$ and the sum is over the
$(n-1)$-cells $z\subset\ov{y}$. Here,
$[y,z]=\pm 1$ is the incidence number of $y$ and $z$ given by
the orientations chosen at the end of \S\ref{subsection:boolean}.
It is easy to check that these homomorphisms assemble into a morphism
of presheaves 
$\delta_n:P_n\rightarrow P_{n-1}$.
The sequence
\begin{equation*}
\cdots
\stackrel{\delta}{\longrightarrow}
P_{n+1}
\stackrel{\delta}{\longrightarrow}
P_{n}
\stackrel{\delta}{\longrightarrow}
P_{n-1}
\stackrel{\delta}{\longrightarrow}
\cdots
\end{equation*}
is exact at $P_n$ if and only if each of the local sequences
$P_*(x)$ is exact at $P_n(x)$. But $P_*(x)$ is
nothing other than the cellular chain complex of the
$\dim(x)$-dimensional ball corresponding to the closure of
$x$ with the induced CW decomposition. 
In particular 
$$
H_nP_*(x)=
\left\{
\begin{array}{ll}
\Z, & n=0\\
0,  & n>0
\end{array}
\right.
$$
so that $P_*(x)$, and hence $P_*$, is exact in degree $n>0$.

To define an augmentation $P_0\stackrel
{\ve}{\rightarrow}\Delta\Z\rightarrow 0$ take $\ve$ to be the
canonical surjection onto $\coker(\delta)$:
$$
P_1\stackrel
{\delta}{\rightarrow}  P_0\stackrel
{\ve}{\rightarrow}\coker(\delta)\rightarrow 0.
$$
The computation of $P_*(x)$  above immediately shows that $\coker(\delta) \cong \Delta\Z$.  

We now have our projective resolution  (\ref{eq:1.2.5})  for
$\Delta\Z$ and  hence a cochain
complex (\ref{eq:1.2.6}) that computes the derived functors $\invlim^i
F_{KH}$. Proposition \ref{subsection:projective:prop1} gives 
an isomorphism of graded abelian groups 
$f:\Hom_{\prsh(\BQ)}(P_*,F_{KH})\rightarrow K^*$ where $K^*$ is
the Khovanov cochain complex of \S\ref{subsection:boolean}. As the
following lemma shows, $f$ is in fact a chain map and thus there is an
isomorphism of cochain complexes
$$
\Hom_{\prsh(\BQ)}(P_*,F_{KH})\cong K^*.
$$

\begin{lemma}
$f$ is a chain map $\Hom_{\prsh(\BQ)}(P_*,F_{KH})\rightarrow
K^*$. 
\end{lemma}

\begin{proof}
We must show that the following diagram commutes.

\begin{equation*}
\xymatrix{
&  \Hom_{\prsh(\catC)} (P_{n+1},F_{KH})\ar[d]^{f^{n+1}}
& \Hom_{\prsh(\catC)} (P_n,F_{KH})\ar[l]_-{\delta} \ar[d]^{f^n}
\\
&  K^{n+1} 
&  K^n \ar[l]_-{d} 
}  
\end{equation*}

Let $\tau\in\Hom_{\prsh(\BQ)}(P_n,F_{KH})$ and write $F$ for $F_{KH}$. If
$x$ is an $n$-cell, write $\lambda_x:=\tau_x(x)\in F(x)$ so that
 $f^n$ sends $\tau$ to the tuple $\sum_x\lambda_x$, the sum
over the $n$-cells of $X$. Applying the Khovanov
differential $d$ we get 
$$
d(f^n(\tau)) = \sum_x \sum_{y\prec\,x} [x,y]\,F(y\prec x)(\lambda_x).
$$ 
Consider now $\delta(\tau) = \tau\delta \in \Hom_{\prsh(\catC)}
(P_{n+1},F_{KH})$. For $y$ an $(n+1)$-cell we have 
$\delta_y(y) = \sum_{x\,\succ y} [x,y]x$ and by an argument similar to
that in  the proof of
Proposition \ref{subsection:projective:prop1}, for $x$ an $n$-cell we
have
$\tau_y(x) = F(y\prec x)(\lambda_x)$. Thus $f^{n+1}(\delta(\tau))$ is
equal to
$$
\sum_{\dim y = n+1}\hspace{-3mm}(\tau_y\delta_y)(y) =
\sum_y \sum_{x\,\succ y} [x,y]\tau_y(x) = \sum_y \sum_{x\,\succ y}
[x,y]F(y\prec x)(\lambda_x) = d(f^n(\tau)).\proved
$$
\end{proof}

Summarizing: to compute the higher limits of the
Khovanov presheaf we use the complex
$\Hom_{\prsh(\BQ)}(P_*,F_{KH})$, which  is isomorphic to
$K^*$, and this in turn computes the unnormalised Khovanov homology.
We have therefore proved the first theorem:

\begin{theorem}
\label{thm:invlim}
Let $D$ be a link diagram and let $F_{KH} \colon \BQop \ra \ab $ be
the Khovanov presheaf defined in \S\ref{subsection:boolean}. Then
 $$
  \bkh i ({D}) \cong {\invlim_{\BQop}}^i F_{KH}
$$
\end{theorem}

\begin{remark}
It is essential that we use the modified Boolean lattice
$\BQ$ rather 
than just $\bB$: if we work with the Khovanov
presheaf over 
$\bB$ then the higher limits all vanish. This
follows from the general fact that for a presheaf over a finite
poset with unique maximal element  the higher limits all vanish -- see
Mitchell \cite{Mitchell72}.
\end{remark}

\subsection{Aside on the cohomology of classifying spaces with
  coefficients in a presheaf}
\label{subsection:classifyingspaces}

Although not central to what follows it is worthwhile
making the connection with a more topological description of 
higher limits in which ${\invlim}^iF_{KH}$ is
identified with the cohomology of a classifying
space equipped with a system of local coefficients.
We recall that the classifying space $B\catC$ is the geometric
realization of the nerve of the small category $\catC$.
This point of
view is novel in the context of  Khovanov homology,
so we give a brief presentation of it, but otherwise we make no particular claim to originality here. 

Starting with a presheaf $F\in\prsh(\catC)$,
the cochain complex $C^*(B\catC,F)$ is defined on the nerve of
$\catC$ to have cochains
$$
C^n(B\catC,F) =\kern-3mm\prod_{x_0\rightarrow\cdots\rightarrow
  x_n}\kern-2mm F(x_0),
$$
the product over sequences of morphisms
$x_0\stackrel{f_1}{\rightarrow} \cdots \stackrel{f_n}{\rightarrow}
x_n$ in $\catC$.
If $\lambda\in C^n$ write 
$\lambda\cdot(x_0\rightarrow\cdots\rightarrow x_n)$ for the component
of $\lambda$  in the copy of $F(x_0)$ indexed by the sequence
$x_0\rightarrow\cdots\rightarrow x_n$.
The coboundary map $d:C^n(B\catC,F)\rightarrow C^{n+1}(B\catC,F)$ is given by
\begin{align*}
d\lambda&\cdot(x_0\stackrel{f_1}{\rightarrow} \cdots \stackrel{f_{n+1}}{\rightarrow} x_{n+1})
=F(x_0\stackrel{f_1}{\rightarrow} x_1)
(\lambda\cdot(x_1\stackrel{f_2}{\rightarrow} \cdots \stackrel{f_{n+1}}{\rightarrow} x_{n+1}))\\
&+\sum_{i=1}^{n} (-1)^i \lambda
\cdot(x_0\stackrel{f_1}{\rightarrow} \cdots x_{i-1} \stackrel{f_{i}f_{i+1}}{\longrightarrow} x_{i+1}\cdots
\stackrel{f_{n+1}}{\rightarrow} x_{n+1})  
+(-1)^{n+1}\lambda\cdot(x_0\stackrel{f_1}{\rightarrow} \cdots \stackrel{f_n}{\rightarrow} x_{n}).
\end{align*}
Write $H^*(B\catC, F)$ for the cohomology of $C^*(B\catC,F)$.
The following result of Moerdijk \cite[Proposition II.6.1]{Moerdijk95}
shows that this cochain complex computes the higher limits: 

\begin{proposition}
\label{prop:limcat}
Let $F\in\presh{\catC}$. Then
${\displaystyle H^*(B\catC, F)\cong {\invlim_{\catC}}^* F}$.
\end{proposition}

From Theorem \ref{thm:invlim} we immediately get the following
description 
of unnormalised Khovanov homology in terms of the cohomology of 
the classifying space $B\BQ$ with a system of local coefficients 
induced by the Khovanov presheaf:

\begin{proposition}  
\label{prop:class}
Let $D$ be a link diagram and let $F_{KH}\colon \BQop \ra \ab $ be
the Khovanov presheaf defined in \S\ref{subsection:boolean}. 
Then, 
$$
   \bkh * ({D}) \cong H^*(B\BQ; F_{KH}).
$$
\end{proposition} 

\begin{remark}
Proposition \ref{prop:class} is very similar in spirit to the main
result (Theorem 24) of the authors \cite{EverittTurner1} which gives an
isomorphism between a homological version of Khovanov homology and a
slight variation on the homology of a poset with coefficients in a
presheaf (termed ``coloured poset homology'' in
\cite{EverittTurner1}; see also \cite{EverittTurner2}). 
\end{remark}

\section{Interpreting higher limits in homotopy theoretic terms}
\label{section:holims}

\subsection{Homotopy limits}
Limits and colimits exist in the category of spaces but are problematic in the
homotopy category: deforming the
input data up to homotopy  may not result in the same homotopy
type. This problem is resolved by the use  of  {\em homotopy}
limits and {\em homotopy} colimits, which are now standard
constructions in homotopy theory.  
In this section we will use homotopy limits to build spaces whose
homotopy groups are Khovanov homology. We begin by recalling the key
properties of homotopy limits, and while we will adopt a blackbox
approach to
the actual construction (leaving the inner workings
firmly inside the box), we will provide references to the
classic text by Bousfield and Kan \cite{BousfieldKan}. 

We briefly return to the generality of a small category, but later
will again specialise to posets. Let $\spaces$ denote the category of
{\em pointed} spaces. All spaces from now on will be pointed. Let
$\BC$ be a small category and let 
$\spaces^\BC$ be the category of diagrams of spaces of shape $\BC$: an
object is a (covariant) functor $\diagX\colon \BC \ra \spaces$
and a morphism $f:\diagX \ra \diagY$ is a natural  transformation.
Thus a diagram of spaces associates to each object of $\BC$ a
(pointed) space and to each morphism of $\BC$ a (pointed) continuous
function such that these fit together 
in a coherent way.  Given a morphism $f:\diagX \ra \diagY$ we will use
the notation 
$f_x$ for the component at $x$. 
The {\em trivial diagram} takes value the
one-point space $\star$  for all objects of $\BC$ and the identity map
$\star\rightarrow\star$ for all morphisms. 

For our purposes holim is a covariant functor
$$
\holim_\BC \colon \diagc \ra \spaces
$$
whose main properties are recalled below in Propositions
\ref{prop:holimhomotopy}-\ref{prop:holimmapping} . For a morphism
$f\colon \diagX\ra \diagY$ we denote by $\bar{f}$ the induced map $\holim \diagX \ra \holim
\diagY$. The holim construction is natural
with respect to change of underlying category: a functor $F\colon \BC^\prime \ra
\BC$ induces a map $ \holim_\BC \diagX \ra
\holim_{\BC^\prime}\diagX \circ F$.

\begin{remark}
We adopt the convention of Bousfield and Kan \cite{BousfieldKan} where if 
pressed on the matter, space means ``simplicial
set''.  Furthermore, if thus
pressed, we will also assume that diagrams take as values  {\em fibrant}
simplicial sets \cite[VIII, 3.8]{BousfieldKan}. Indeed there are
models of Eilenberg-Mac Lane spaces that are simplicial groups, and
hence fibrant. The reader should be aware however that in the proper
generality the propositions below require fibrant objects.
\end{remark}

The first important property of holim is its well-definedness in the homotopy
category; it is robust with respect
to deformation by homotopy \cite[XI, 5.6]{BousfieldKan}:

\begin{proposition}[Homotopy]
\label{prop:holimhomotopy}
Let $f\colon \diagX \ra \diagY$ be a morphism in $\diagc$ such that
for all $x\in\diagX$ the map 
$f_x\colon \diagX(x) \ra \diagY(x)$ is a homotopy equivalence. Then      
$\bar{f}\colon \holim \diagX \ra \holim \diagY$ is a homotopy
equivalence.  
\end{proposition}

Next, a morphism of diagrams which is locally a fibration  induces a
fibration on holim \cite[XI, 5.5]{BousfieldKan}:

\begin{proposition}[Fibration]
\label{prop:holimfibration}
Let $f\colon \diagX \ra \diagY$ be a morphism in $\diagc$ such that
for all $x\in\diagX$ 
the map $f_x\colon \diagX(x) \ra \diagY(x)$ is a fibration. Then      
$\bar{f}\colon \holim \diagX \ra \holim \diagY$ is a fibration. 
\end{proposition}

There is also a nice description of holim for diagrams over a
product of categories \cite[XI, 4.3]{BousfieldKan}:

\begin{proposition}[Product]
\label{prop:holimproduct}
Let $\diagX\colon \BC \times \BD\ra \spaces$ be a diagram of spaces
over the product category $\BC \times \BD$. Then
$$
\holim_\BC \holim_\BD \diagX \simeq \holim_{\BC\times\BD}\diagX
\simeq\holim_\BD \holim_\BC \diagX
$$
\end{proposition}

We also need to be able to compare diagrams of different
shape, i.e.: where the base categories are different. The result
turns out to be easier to state in the context of posets than for
small categories, and this suffices for us \cite[XI, 9.2]{BousfieldKan}:

\begin{proposition}[Cofinality]
\label{prop:holimcofinality}
Let $f\colon \BP_2 \ra \BP_1$ be a map of posets. 
\begin{enumerate}
\item[(i)] Let $\diagX\colon
\BP_1\ra \spaces$ 
be a diagram of spaces 
and suppose
that for any $x\in \BP_1$ 
the poset $f^{-1}\{y\in \BP_1 \mid y\leq x\} \subset \BP_2$ is contractible. Then
$\holim_{{\BP_2}}\diagX\circ f \simeq  \holim_{\BP_1}
\diagX$.
\item[(ii)] Let $\diagX\colon
\BP_1^\op\ra \spaces$ 
be a diagram of spaces 
and suppose that for any $x\in \BP_1$ 
the poset $f^{-1}\{y\in \BP_1 \mid y\geq x\} \subset \BP_2$ is contractible. Then
$\holim_{{\BP_2}^\op}\diagX\circ f \simeq  \holim_{\BP_1^\op}
\diagX$.
\end{enumerate}
\end{proposition}

Here a poset is contractible if its
geometric realisation $B\BP$ is contractible, so
in particular $B\BP$, and hence $\BP$, is non-empty. 
For example if $\BP$ has an extremal (i.e.: maximal or minimal)
element then $B\BP$ is a cone. Statement  (ii) above is simply a
restatement of (i), but the potential confusion in taking
opposites makes it worth while stating both.

For a simple application of Proposition \ref{prop:holimcofinality}
let $\BP_2$ be a contractible poset and
$\diagX$ the constant diagram over $\BP_2$ having value the space $X$
at each $x$ and the identity map $X\ra X$ at each morphism $x\ra
y$. Let $\BP_1$ be the single element poset and $\diagY$ the diagram
having value $X$ at this single element. If $f:\BP_2\ra\BP_1$ is the
only possible map, then $\diagX=\diagY\circ f$ and the conditions of
Proposition \ref{prop:holimcofinality} are satisfied. Thus
$\holim\diagX\simeq\holim\diagY\simeq X$.


The final basic property of holim is that it commutes with mapping
spaces (of pointed maps between pointed spaces) -- see \cite[XI, 7.6]{BousfieldKan}:

\begin{proposition}[Mapping]
\label{prop:holimmapping}
Let $\diagX$ be a diagram of spaces in $\ptdiagc$ and let $Y$ be a (pointed) space. Then
$$
\Map(Y, \holim\diagX)
\simeq 
\holim \Map (Y,\diagX).
$$
\end{proposition}
Here $\Map(Y, -)$ is the functor that takes a pointed space $Z$ to
the space of pointed maps from $Y$ to $Z$ and $\Map (Y,\diagX)\in
\ptdiagc$ is the composition $\Map(Y, -)\circ \diagX$.

\subsection{Spaces for Khovanov homology}\label{sec:spaces}

Bousfield and Kan give an interpretation of derived functors of the
inverse limit as follows. Consider the
Eilenberg-Mac Lane functor $K(-,n):\ab\ra\spaces$ for which we adopt the
construction given by Weibel \cite[8.4.4]{Weibel94} where 
there is an obvious
choice of basepoint for $K(A,n)$. For more details on Eilenberg-Mac Lane
spaces see May \cite[Chapter V]{May} or Hatcher \cite[Chapter 4]{Hatcher}.
The following proposition \cite[XI, 7.2]{BousfieldKan} gives an interpretation of
$\invlim^i F$ in homotopy theoretic terms where $\BC$ is a small category.

\begin{proposition}
\label{prop:piholim}
Let $F\colon \BC \ra \ab$ be a (covariant) functor.
Then there are natural isomorphisms:
$$
\pi_i(\holim_\BC K(-,n) \circ F) \cong 
\begin{cases}
\invlim^{n-i}_\BC F& 0\leq i \leq n,\\
0 & \text{ else.}
\end{cases}
$$
\end{proposition}

The spaces $\holim_\BC K(-,n) \circ F $ contain no more information
than higher derived functors of $F$.
 Indeed, as a consequence of the Dold-Kan theorem (see
 \cite[Section 8.4]{Weibel94} or Curtis \cite[Section 5]{Curtis71}) we have:

\begin{proposition}
\label{prop:prodEM}
 For $n$ big enough the space $\holim_\BC K(-,n) \circ F $ has the
 homotopy type of a product of 
Eilenberg-Mac Lane spaces: 
$$
\holim_\BC K(-,n) \circ F  \simeq \prod_m K( {\invlim}^{n-m}_\BC F, m).
$$
\end{proposition}

For a self contained and elementary argument proving the appropriate
result needed here we refer the reader to  
\cite{EverittLipshitzSarkarTurner}.

After these preliminaries on homotopy limits we return to
Khovanov homology. Associated to a link diagram $D$ we have the Khovanov presheaf  
$F_{KH} \colon \BQop \ra \ab $ of \S\ref{subsection:boolean}. Let $n\in\bN$ 
and let $\scrF_n\colon
\BQop \ra
\ptspaces $ be the diagram of spaces defined by
$\scrF_n=K(-,n)\circ F_{KH}$, the composition of 
$F_{KH}$ with the Eilenberg-Mac Lane space functor  $K(-,n)$.
We can now define a space
$\YYn D$ as the homotopy limit of this diagram:
$$
\YYn D = \holim_{\BQop}\scrF_n=\holim_{\BQop} K(-,n)\circ F_{KH} . 
$$

\begin{remark}\label{rem:goodwillie}
The homotopy limit above, taken over the augmented Boolean
lattice $\BQ$, is what Goodwillie \cite{Goodwillie92}, in his theory of calculus of
functors, calls the {\em total fiber} of the
(decorated) Boolean lattice $\bB$.  
\end{remark}

From Theorem \ref{thm:invlim}  and Proposition \ref{prop:piholim} we
see that 
$\YYn D$ is a space whose homotopy groups are isomorphic to the
unnormalised Khovanov homology of $D$:

\begin{proposition} 
\label{prop:homotopygroups}
\begin{equation}
  \label{eq:5}
  \pi_i(\YYn {D}) \cong 
\begin{cases}
\bkh {n-i} ({D}) & 0\leq i \leq n,\\
0 & \text{ else.}
\end{cases}
\end{equation}
\end{proposition}

Indeed by Proposition \ref{prop:prodEM} we have
$$
\YYn D \simeq \prod_{m} K(\bkh {n-m}(D),m).
$$

In order to normalise Khovanov homology a global degree shift is
applied. As $\pi_{i} \Omega X\cong \pi_{i+1}X$ for a pointed space $X$, 
we see that degree shifts are implemented at the space
level by taking loop spaces.  Suppose
now $D$ is oriented and has $c$ negative crossings.
The collection of  spaces $\YY = \{ \YYn D\}$  is an $\Omega$-spectrum
which may be delooped $c$ times to obtain a new $\Omega$-spectrum
$\XX D= \Omega^{-c}\YY D$ whose
homotopy groups are normalized Khovanov homology:
$$
\pi_i(\XX D) \cong KH^{-i}(D).
$$

\section{Diagrams over Boolean lattices and homotopy limits}

This section develops some results on the homotopy limits of diagrams defined over
(modified) Boolean lattices. These will then provide tools applicable to
Khovanov homology, and in the next section we illustrate this with
homotopy theoretic proofs of some Khovanov
homology results. In light of the Remark after Prosposition \ref{prop:prodEM} some of the
conclusions of this
section are consequences of Goodwillie's calculus of functors, but we
prefer to (re)prove the results we need in a self-contained manner. 

We make extensive use of homotopy fibres and so record here some of their
properties. Given a map (of pointed spaces)
$f\colon X\ra Y$ we define the  {\em homotopy
   fibre} of $f$ as a homotopy limit by
\begin{eqnarray}
  \label{eq:2}
\hofibre(X\stackrel{f}{\longrightarrow}Y)
=
\holim(X\stackrel{f}{\longrightarrow}Y\longleftarrow\star).
\end{eqnarray}

By lifting the lid of the black box only a fraction (see
\cite[Chapter XI]{BousfieldKan}) one sees that this has the homotopy 
type of the usual homotopy fibre: namely defining
\begin{equation}
  \label{eq:9}
E_f=\{ (x,\alpha) \mid x\in X,  \alpha \colon [0,1]\ra Y \text{ a
  continuous map such that } \alpha(0) = f(x)\},
\end{equation}
then this is a space homotopy equivalent to $X$ and the map  $E_f\ra Y$ sending
 $(x,\alpha) \mapsto \alpha(1)$ is a fibration whose fibre is homotopy
equivalent to the hofibre (\ref{eq:2}).

Relevant examples of homotopy fibres are
\begin{eqnarray}
  \label{eq:1}
  \hofibre (X\ra \star) \simeq X 
\end{eqnarray}
\begin{eqnarray}
  \label{eq:3}
  \hofibre (\star \ra Y)\simeq \Omega Y.
\end{eqnarray}

Using the long exact homotopy sequence for a fibration and the
Whitehead 
theorem one immediately gets:

\begin{lemma}\label{lem:hofibre}
For $Y$ connected, if $\hofibre (X\ra Y) \simeq \star$ then $X\simeq Y$. 
\end{lemma}

If $f\colon X\ra Y$ is a map of pointed spaces with $Y$ 
contractible,
then (\ref{eq:1}) extends to 
\begin{eqnarray}
  \label{eq:7}
  \hofibre (X\stackrel{f}{\longrightarrow}Y)\simeq X,
\end{eqnarray}
and similarly if $X$ is contractible then (\ref{eq:3})  extends to
 \begin{eqnarray}
  \label{eq:6}
   \hofibre (X\stackrel{f}{\longrightarrow}Y)\simeq \Omega Y.
\end{eqnarray}

From now on we assume that $\catC$ is a connected category.  
Given diagrams $\diagX, \diagY\in\ptdiagc$ (of pointed spaces) and a morphism 
$f\colon \diagX\ra \diagY$ one may form the homotopy fibre diagram 
$\hof(f)$ by (locally) defining $\hof(f)(x) = \hofibre (f_x\colon
\diagX(x) \ra \diagY(x))$
and $\hof(f)(x\ra y):\hofibre(f_x)\ra\hofibre(f_y)$ the map induced by
taking homotopy limits of the two rows of the diagram
$$
\xymatrix{\diagX(x) \ar[r]\ar[d] & \diagY(x)\ar[d] & \star \ar[l]\ar[d]\\\diagX(y) \ar[r] & \diagY(y) &
  \star \ar[l]}
$$
with the lefthand square commuting courtesy of $f$. 
As the next result shows, one can identify the homotopy limit of $\hof(f)$ with the homotopy
fibre 
of the the map $\bar{f}$. In the interest of completeness we have included the details, but the main point is that the homotopy fibre is an example of a homotopy limit and homotopy limits enjoy the Product property of Proposition \ref{prop:holimproduct}:

\begin{proposition} \label{prop:hofibre}
 Let $f\colon \diagX\ra \diagY$ be a morphism in $\ptdiagc $. Then
$$
\holim (\hof(f)) 
\simeq 
\hofibre (\holim\diagX\stackrel{\bar{f}}{\longrightarrow}\holim\diagY).
$$
\end{proposition}

\begin{proof}
Let $\BD$ be the three element category 
$a\stackrel{\aa}{\longrightarrow}c\stackrel{\beta}{\longleftarrow}b$.
Define 
$\diagZ\colon \BC\times\BD\ra \ptspaces$ by
$$
\diagZ(x,a) = \star \;\;\;\;\;\;\;\; \diagZ(x,b) = \diagX (x) \;\;\;\;\;\;\;\;\diagZ(x,c) = \diagY (x).
$$
On morphisms let $\diagZ((id,\alpha)\colon (x,a) \ra (x, c)) = \star
\ra \diagY(x)$,
$\diagZ((id,\beta)\colon (x,b) \ra (x, c)) =  f_x$
and
$$
\diagZ((\theta,z\stackrel{1}{\ra}z)\colon (x,z) \ra (x^\prime, z)) 
= 
\left\{
\begin{array}{ll}
\star\ra\star,&z=a,\\
\diagX (\theta),&z=b,\\
\diagY (\theta),&z=c.
\end{array}\right.
$$
We then have
\begin{eqnarray*}
 \holim_{\BD} \diagZ(x,-) & = & 
\holim (
\diagX(x)\stackrel{f_x}{\longrightarrow}\diagY(x)\longleftarrow\star
)\\
& =&\hofibre (f_x\colon\diagX(x) \ra \diagY(x))\;\;\;\;\text{ [by (\ref{eq:2})]}\\
& = & \hof (f)(x)
\end{eqnarray*}
from which we get
$\holim_{\BC}\holim_{\BD} \diagZ \simeq \holim\hof(f)$.
Going the other way we have 
$
\holim_{\BC}\diagZ(-,a)  \simeq \star$, 
$\holim_{\BC}\diagZ(-,b) = \holim_{\BC} \diagX$, and 
$\holim_{\BC}\diagZ(-,c) = \holim_{\BC} \diagY,
$
so
\begin{eqnarray*}
  \holim_{\BD} \holim_{\BC} \diagZ  & = &  
\holim (
\holim\diagX\stackrel{\bar{f}}{\longrightarrow}\holim\diagY\longleftarrow\star
)\\
& \simeq & \hofibre (\holim \diagX\stackrel{\bar{f}}{\longrightarrow}\holim \diagY).
\end{eqnarray*}
The result now follows from Proposition \ref{prop:holimproduct}.
\end{proof}


Later we will  use this result in the form: 
if $f:\diagX\rightarrow\diagY$ is a map of diagrams, the space $\holim\diagY$ is connected and  $\holim\hof(f)$ contractible, then the induced map 
$\bar{f}:\holim\diagX\rightarrow\holim\diagY$ is a homotopy
equivalence.

\paragraph{Notation for diagrams of spaces.} 
We introduce a convenient notation that we will use extensively.
We recall that space means {\em pointed } space and diagrams of spaces
take values in pointed spaces. A Boolean lattice $\bB$ will be represented by the
circle below left and a diagram 
$\diagX \colon\bB^\op \ra \ptspaces$ by
the pictogram below right:
$$

$$
Extending this to $\BQop$ by defining $\diagX (\1^\prime)=\star$ we
obtain a diagram 
of spaces, with $\BQ$ and 
$\diagX \colon\BQop \ra \ptspaces$ represented as
$$
%
$$

The trivial diagram 
will be denoted
$%
$.

Let $\bB=\bB_A$ be Boolean of rank $r$, i.e.: the lattice of subsets
of $\{1,\ldots,r\}$. For each $1\leq k\leq r$ there is a splitting of
$\bB$ into two subposets, both isomorphic to Boolean lattices of rank
$r-1$: one consists of those subsets containing $k$ and the other of 
those not containing $k$. 
Below we see the splittings for $r=3$, with (from left
to right) $k=1,2$ and $3$:
$$
%
$$
A diagram $\diagX \colon\bB^\op \ra \ptspaces$ 
determines (and is determined by)  two diagrams of spaces $\diagX_1$
and $\diagX_2$ over these rank $r-1$ Boolean lattices along with a
morphism of 
diagrams $f\colon\diagX_1 \ra \diagX_2$. We denote this situation (and
the obvious extension to 
$\BQop$) by the pictograms:
$$
%
$$
This process can be iterated with each of the smaller Boolean lattices
to give pictures that are square, cubical, etc. 


\begin{lemma}
\label{lem:holim1} 
Let $\diagX\colon \bB^\op \ra
  \ptspaces$ be a diagram of spaces. Then,
$$
%
$$
\end{lemma}

\begin{proof}
If $\bB$ has rank $r$ then the diagram on the left-hand side is over a
Boolean lattice of rank $r+1$.
Let $\BQ$ be the extended version of this Boolean lattice 
and suppose that it has been split as above. 
Collapsing the bottom Boolean lattice to a point we obtain a new
poset $\BP$:
$$
%
$$
The poset map $f\colon\BQ\ra\BP$ which collapses the lower Boolean
lattice to a single point satisfies the hypotheses of Proposition
\ref{prop:holimcofinality} (ii). 
Moreover we have the following equality of diagrams of spaces (of shape $\BQ^\op$):
$$
%
$$
Let
$\BQ^\prime$
be the poset obtained from $\BQ$ by removing the lower Boolean lattice:
then the obvious inclusion $i\colon\BQ^\prime \ra \BP$
satisfies the
hypotheses of Proposition \ref{prop:holimcofinality}. Hence we have
$$
%
$$
using Proposition \ref{prop:holimcofinality} twice
(with $f$ and $i$).
\end{proof}

A somewhat more general version of this result is the following.

\begin{lemma}\label{lem:holim1A} Let $\diagX, \diagY\colon \bB^\op \ra
  \ptspaces$ be diagrams of spaces such that $\diagY(x)$ is contractible 
for all $x\in \bB^\op$. Then,
$$
%
$$
  \end{lemma}

  \begin{proof}
Let $\tau$ be the morphism of diagrams defined by
$$
%
$$
As the map $\diagY(x) \ra \star$ is a homotopy
equivalence for all $x$, 
the result follows from Proposition \ref{prop:holimhomotopy} and Lemma \ref{lem:holim1}.    
\end{proof}

\begin{lemma}\label{lem:holim2} 
Let $\diagX\colon \bB^\op \ra
  \ptspaces$ be a diagram of spaces. Then,
$$
%
$$
\end{lemma}

\begin{proof}
Let $\BQ$ be an extended Boolean lattice of rank one bigger than the
rank of $\bB$ and split as above.
Let $\BP$  be obtained from $\bB$ by adding an element $\1^{\prime\prime}$
which is greater than all other elements (including the existing
maximal element in $\bB$). Pictorially:
$$
%
$$
The poset map $f\colon\BQ\ra\BP$ which identifies elements 
of the Boolean lattices and sends $\1^\prime \mapsto \1^{\prime\prime}$ 
satisfies  the hypotheses of Proposition \ref{prop:holimcofinality}. 
Moreover we have the following equality of diagrams of spaces (of shape $\BQ^\op$).
$$
%
$$
Since $\BP^{\op}$ has a minimal element it follows from 
\cite[XI, 4.1(iii)]{BousfieldKan} that 
$$
%
$$
whence the result on applying Proposition \ref{prop:holimcofinality}.
\end{proof}

\begin{proposition}\label{prop:hohoho}
Let 
$f\colon
%
$
be a morphism of diagrams of  spaces over a Boolean lattice. Then,
$$
%
$$
\end{proposition}

\begin{proof}
Let $g$ be the following morphism of diagrams of spaces
$$
%
$$
We have $\holim(\hof(g))$ is homotopy equivalent to
$$
%
$$
by Proposition \ref{prop:hofibre}, Lemma \ref{lem:holim2} and
(\ref{eq:7}). On the other hand writing
$$
%
$$
we see that $\diagZ_2(x)$ is contractible for all $x$. 
Thus, by Lemma \ref{lem:holim1A} and Proposition \ref{prop:hofibre} we have
$$
%
$
be a morphism of diagrams of  spaces over a Boolean lattice. Then,
$$
%
$$  
\end{corollary}

Lemma \ref{lem:holim1}
can be generalized as in the first part of the following:

\begin{proposition}\label{prop:holim4}
 Let $\diagY\colon \bB^\op \ra \ptspaces$ be a diagram of spaces with
$\holim %
$$
by Proposition \ref{prop:hohoho} and Equation (\ref{eq:6}).
Part (i) is similar.
\end{proof}

\begin{corollary}\label{cor:holim4B}
Let $\diagX, \diagY\colon \bB^\op \ra
\ptspaces$ be diagrams of  spaces such that $\diagY(x)$ is contractible 
for all $x\in \bB^\op$. Then,
$$
%
$$
  \end{corollary}

  \begin{proof}
Let $\tau$ be the morphism of diagrams defined by
$$
%
$$
As the map $\star \ra \diagY(x) $ is a homotopy
equivalence for all $x$, 
the result follows from Propositions \ref{prop:holimhomotopy} and 
\ref{prop:holim4}(ii).    
\end{proof}

\paragraph{Notation for presheaves.} 
We adopt a similar notation for presheaves to that 
for diagrams 
of spaces, with the difference that the circles are white rather than
shaded. 
Thus
a presheaf $F\colon\bB^\op \ra \ab$ and its extension  
$F\colon\BQop \ra \ab$  (with  $F (\1^\prime)=0$)
will be represented by:
$$
%
$ 
we will denote the diagram of spaces $ K(-,n)\circ F $ by 
$
%
$.

\begin{remark}
\label{rem:connected}
 For $n$ sufficiently large the space 
$\holim 
%
$
is path connected. To see this, one may use Proposition \ref{prop:piholim} to calculate 
$$
\pi_0(
\holim 
%
)\cong\textstyle{\invlim^{n}} F.
$$
So long as $n$ is chosen to be greater than the rank of $\BQ$ the
right-hand side is trivial
by Theorem \ref{thm:invlim} (recall that the underlying poset is
always assumed connected). 
We will always assume that $n$ is large enough in this sense.
\end{remark}

\begin{proposition} 
\label{prop:holimhofibre}
Let 
$
%
$ 
be a short exact sequence in  $\preshq$. Then,
\begin{align*}
&(i).\;\;\hofibre(
\holim 

\end{align*}
\end{proposition}

\begin{proof}
We will prove part (i). Let $T$ be the morphism of presheaves below left inducing the
morphism of diagrams $\tau$ below right:
$$
%
$$
By the naturality of the isomorphism in Proposition \ref{prop:piholim}
the map
$$
\ov{\tau}_*:\pi_i(\holim
\left(
%
\right)
)
$$
can be identified with the map
$$
T_*:{\textstyle \invlim^{n-i}}
\left(
%
\right)
$$
which is an isomorphism courtesy of the long exact sequence
(\ref{eq:10}) of higher limits coming from the short exact sequence
of presheaves:
$$
%
$$
Thus since $\ov{\tau}_*$ is an isomorphism the Whitehead Theorem gives
that $\ov{\tau}$ is a homotopy equivalence, and so
$\holim
%
$ is homotopy equivalent to
$$
)
$$
by Lemma \ref{lem:holim1}, the isomorphism $\ov{\tau}$ and 
Proposition \ref{prop:hohoho}. The proof of part (ii) is completely 
analogous.
\end{proof}

\noindent Combining Proposition \ref{prop:holimhofibre}, Lemma
\ref{lem:hofibre} and the Remark preceeding Proposition \ref{prop:holimhofibre}
gives the following very useful lemma. 

\begin{corollary}\label{section3:corollary200}
Let 
$
%
$ 
be a short exact sequence in  $\preshq$. Then,
\begin{align*}
  \label{eq:8}
&(i).\text{ If }\holim
%
  
\text{ is a homotopy equivalence.}
\end{align*}
\end{corollary}

\noindent Another result that will prove useful
comes from
combining Propositions \ref{prop:holimhofibre} and \ref{prop:hohoho}:

\begin{corollary}\label{cor:holimhofibre}
Let 
$
%
$ 
be a short exact sequence in  $\preshq$. Then,
$$
(i).\,\holim\left(
%
$$
\end{corollary}

\section{Applications to Khovanov homology}
\label{section:applications}

The results of the previous section give a collection of tools
for Khovanov homology, and in this section we illustrate
with a few simple examples. First we isolate the two most
useful results. One is Corollary \ref{section3:corollary200},
which is just a presheaf theoretic reformulation of standard arguments
involving chain complexes:

\begin{firstdevice}
\label{device:1}
Let 
$
%
$ 
be a short exact sequence.
If 
$\holim%
$
)
is contractible then 
$$
\holim
%
)
$$
\end{firstdevice}


An arbitrary morphism of presheaves (not necessarily injective or
surjective) cannot be slotted into a short exact sequence. Our second
result, which follows from Proposition \ref{prop:hofibre} and the
discussion preceeding it along with Remark \ref{rem:connected}, gets around this by using
the homotopy theory in a more 
essential way:

\begin{seconddevice}
\label{device:2}
Let 
$
\begin{pspicture}(0,0)(0.8,0.5)
\rput(0.4,-0.35){\rput(0,0.5){$F$}\pscircle(0,0.5){0.35}
\pscircle[fillstyle=solid,fillcolor=white](0.185,0.765){0.075}}
\end{pspicture}
\ra
\begin{pspicture}(0,0)(0.8,0.5)
\rput(0.4,-0.35){\rput(0,0.5){$G$}\pscircle(0,0.5){0.35}
\pscircle[fillstyle=solid,fillcolor=white](0.185,0.765){0.075}}
\end{pspicture}
$ 
be a morphism in $\preshq$ with 
$
\begin{pspicture}(0,0)(0.8,0.5)
\rput(0.4,-0.35){
\pscircle[fillstyle=solid,fillcolor=lightergray](0,0.5){0.35}
\pscircle[fillstyle=solid,fillcolor=white](0.185,0.765){0.075}
\rput(0,0.5){$\scrH$}
}
\end{pspicture}
$
$
=\hof(
\begin{pspicture}(0,0)(0.8,0.5)
\rput(0.4,-0.35){
\pscircle[fillstyle=solid,fillcolor=lightergray](0,0.5){0.35}
\pscircle[fillstyle=solid,fillcolor=white](0.185,0.765){0.075}
\rput(0,0.5){$F$}
}
\end{pspicture}
\ra
\begin{pspicture}(0,0)(0.8,0.5)
\rput(0.4,-0.35){
\pscircle[fillstyle=solid,fillcolor=lightergray](0,0.5){0.35}
\pscircle[fillstyle=solid,fillcolor=white](0.185,0.765){0.075}
\rput(0,0.5){$G$}
}
\end{pspicture}
)
$.
If
$$
\holim\begin{pspicture}(0,0)(0.8,0.5)
\rput(-6.8,-1.1){
\rput(7.2,0.7){
\pscircle[fillstyle=solid,fillcolor=lightergray](0,0.5){0.35}
\pscircle[fillstyle=solid,fillcolor=white](0.185,0.765){0.075}
\rput(0,0.5){$\scrH$}
}}
\end{pspicture}
\quad
\text{(resp.}
\holim\begin{pspicture}(0,0)(0.8,0.5)
\rput(-6.8,-1.1){
\rput(7.2,0.7){
\pscircle[fillstyle=solid,fillcolor=lightergray](0,0.5){0.35}
\pscircle[fillstyle=solid,fillcolor=white](0.185,0.765){0.075}
\rput(0,0.5){$G$}
}}
\end{pspicture})
\quad
\text{is contractible}
$$
then 
$$
\holim 
\begin{pspicture}(0,0)(0.8,0.5)
\rput(-6.8,-1.1){
\rput(7.2,0.7){
\pscircle[fillstyle=solid,fillcolor=lightergray](0,0.5){0.35}
\pscircle[fillstyle=solid,fillcolor=white](0.185,0.765){0.075}
\rput(0,0.5){$F$}
}}
\end{pspicture}
\simeq
\holim 
\begin{pspicture}(0,0)(0.8,0.5)
\rput(-6.8,-1.1){
\rput(7.2,0.7){
\pscircle[fillstyle=solid,fillcolor=lightergray](0,0.5){0.35}
\pscircle[fillstyle=solid,fillcolor=white](0.185,0.765){0.075}
\rput(0,0.5){$G$}
}}
\end{pspicture}
\quad
(
\text{resp.}
\holim
\begin{pspicture}(0,0)(0.8,0.5)
\rput(-6.8,-1.1){
\rput(7.2,0.7){
\pscircle[fillstyle=solid,fillcolor=lightergray](0,0.5){0.35}
\pscircle[fillstyle=solid,fillcolor=white](0.185,0.765){0.075}
\rput(0,0.5){$\scrH$}
}}
\end{pspicture}
\simeq
\holim 
\begin{pspicture}(0,0)(0.8,0.5)
\rput(-6.8,-1.1){
\rput(7.2,0.7){
\pscircle[fillstyle=solid,fillcolor=lightergray](0,0.5){0.35}
\pscircle[fillstyle=solid,fillcolor=white](0.185,0.765){0.075}
\rput(0,0.5){$F$}
}}
\end{pspicture}
).
$$
\end{seconddevice}

\paragraph{Notation for the Khovanov presheaf of a link diagram.} 
We extend our notation to the specific case of the Khovanov 
presheaf $F_{KH}$  associated to a link diagram. 
It is convenient to have the link diagram  back in the picture:
given a link 
diagam $D$ we denote the associated Khovanov presheaf $F_{KH}$ and 
diagram of spaces $K(-,n)\circ F_{KH}$ by:
\begin{equation*}
  \label{eq:12}
\begin{pspicture}(0,0)(10,1)
\rput(3.5,0.5){
\rput(0,0){$D$}
\pscircle(0,0){0.45}
\pscircle[fillstyle=solid,fillcolor=white](0.260,0.340){0.090}
}
\rput(5.25,0.5){and}
\rput(7,0.5){
\pscircle[fillstyle=solid,fillcolor=lightergray](0,0){0.45}
\rput(0,0){$D$}
\pscircle[fillstyle=solid,fillcolor=white](0.260,0.340){0.090}
}
\end{pspicture}
\end{equation*}
If, as is often the case, we are interested in a link diagram with a
specified 
local piece we simply display it inside the circle. 
Thus for example
the unit map  $\iota \colon \bZ \ra V$, which is defined by $\iota(1) = 1$,
extends to an injective morphism of presheaves over
a Boolean lattice of appropriate rank which we denote by:
$$
\begin{pspicture}(0,0)(13,1)
\rput(-1,0){
%
\rput(6.25,0){
\pscircle[fillstyle=solid,fillcolor=white](0,0.5){0.45}
\psbezier[showpoints=false](-0.35,0.74)(-0.15,0.65)(-0.1,0.55)(-0.1,0.5)
\psbezier[showpoints=false](-0.35,0.26)(-0.15,0.35)(-0.1,0.45)(-0.1,0.5)
}
\psline{->}(7,0.5)(8,0.5)
\rput(7.5,0.7){$\scriptstyle{\iota}$}
\rput(8.75,0){
\pscircle[fillstyle=solid,fillcolor=white](0,0.5){0.45}
\psbezier[showpoints=false](-0.35,0.74)(-0.15,0.65)(-0.1,0.55)(-0.1,0.5)
\psbezier[showpoints=false](-0.35,0.26)(-0.15,0.35)(-0.1,0.45)(-0.1,0.5)
\pscircle(0.15,0.5){0.15}
}}
\end{pspicture}
$$
We have
link diagrams $D$ and $D'$ that are identical outside one part where
they differ by the local piece shown. On the left we have the presheaf
$F_{KH}\colon\bB^\op\ra\ab$, the cube for $D$, and on the right
$F_{KH}^\prime \colon\bB^\op\ra\ab$, the cube
for $D'$.
For $x\in\bB$ we have $F_{KH}^\prime (x) = F_{KH}(x) \otimes V$ and
the map $\iota_x\colon F_{KH}(x)=  F_{KH}(x)\otimes \Z \ra  F_{KH}(x)
\otimes V = F_{KH}^\prime (x) $ is the map $1\otimes \iota$.
These local
$\iota_x$ stitch together to form a morphism of presheaves
$\iota:F_{KH}\ra F_{KH}'$. This is what we mean by the picture
above.

Similarly the counit map $\epsilon \colon V \ra  \bZ$, defined by
$\epsilon(1) = 0, 
\epsilon(u) = 1$, extends to a surjective morphism of presheaves over $\bB$
$$
\begin{pspicture}(0,0)(13,1)
%
\rput(-1,0){
\rput(6.25,0){
\pscircle[fillstyle=solid,fillcolor=white](0,0.5){0.45}
\psbezier[showpoints=false](-0.35,0.74)(-0.15,0.65)(-0.1,0.55)(-0.1,0.5)
\psbezier[showpoints=false](-0.35,0.26)(-0.15,0.35)(-0.1,0.45)(-0.1,0.5)
\pscircle(0.15,0.5){0.15}
}
\psline{->}(7,0.5)(8,0.5)
\rput(7.5,0.7){$\scriptstyle{\epsilon}$}
\rput(8.75,0){
\pscircle[fillstyle=solid,fillcolor=white](0,0.5){0.45}
\psbezier[showpoints=false](-0.35,0.74)(-0.15,0.65)(-0.1,0.55)(-0.1,0.5)
\psbezier[showpoints=false](-0.35,0.26)(-0.15,0.35)(-0.1,0.45)(-0.1,0.5)
}}
\end{pspicture}
$$
and
there is a short exact sequence of presheaves
\begin{equation}\label{eq:unit}
\begin{pspicture}(0,0)(10,1)
%
\rput(-2.5,0){
\rput(5,0){
\pscircle[fillstyle=solid,fillcolor=white](0,0.5){0.45}
\psbezier[showpoints=false](-0.35,0.74)(-0.15,0.65)(-0.1,0.55)(-0.1,0.5)
\psbezier[showpoints=false](-0.35,0.26)(-0.15,0.35)(-0.1,0.45)(-0.1,0.5)
}
\psline{>->}(5.75,0.5)(6.75,0.5)
\rput(6.25,0.7){$\scriptstyle{\iota}$}
\rput(7.5,0){
\pscircle[fillstyle=solid,fillcolor=white](0,0.5){0.45}
\psbezier[showpoints=false](-0.35,0.74)(-0.15,0.65)(-0.1,0.55)(-0.1,0.5)
\psbezier[showpoints=false](-0.35,0.26)(-0.15,0.35)(-0.1,0.45)(-0.1,0.5)
\pscircle(0.15,0.5){0.15}
}
\psline{->>}(8.25,0.5)(9.25,0.5)
\rput(8.75,0.7){$\scriptstyle{\epsilon}$}
\rput(10,0){
\pscircle[fillstyle=solid,fillcolor=white](0,0.5){0.45}
\psbezier[showpoints=false](-0.35,0.74)(-0.15,0.65)(-0.1,0.55)(-0.1,0.5)
\psbezier[showpoints=false](-0.35,0.26)(-0.15,0.35)(-0.1,0.45)(-0.1,0.5)
}
}
\end{pspicture}
\end{equation}
The multiplication $m\colon V\otimes V\ra V$ is surjective
with $\ker(m)\cong V$, and this analysis similarly extends to give
a short exact sequence of presheaves
\begin{equation}\label{eq:mult}
\begin{pspicture}(0,0)(10,1)
%
\rput(-2.5,0){
\rput(5,0){
\pscircle[fillstyle=solid,fillcolor=white](0,0.5){0.45}
\psbezier[showpoints=false](-0.35,0.74)(-0.05,0.5)(0.2,0.85)(0.25,0.5)
\psbezier[showpoints=false](-0.35,0.26)(-0.05,0.5)(0.2,0.15)(0.25,0.5)
}
\psline{>->}(5.75,0.5)(6.75,0.5)
\rput(7.5,0){
\pscircle[fillstyle=solid,fillcolor=white](0,0.5){0.45}
\psbezier[showpoints=false](-0.35,0.74)(-0.15,0.65)(-0.1,0.55)(-0.1,0.5)
\psbezier[showpoints=false](-0.35,0.26)(-0.15,0.35)(-0.1,0.45)(-0.1,0.5)
\pscircle(0.15,0.5){0.15}
}
\psline{->>}(8.25,0.5)(9.25,0.5)
\rput(8.75,0.7){$\scriptstyle{m}$}
\rput(10,0){
\pscircle[fillstyle=solid,fillcolor=white](0,0.5){0.45}
\psbezier[showpoints=false](-0.35,0.74)(-0.05,0.5)(0.2,0.85)(0.25,0.5)
\psbezier[showpoints=false](-0.35,0.26)(-0.05,0.5)(0.2,0.15)(0.25,0.5)
}
}
\end{pspicture}
\end{equation}
The composition $m\circ \iota $ is the identity map.
Finally, the comultiplication $\Delta\colon V \ra V\otimes V$ is injective,
$\coker(\Delta)\cong V$, and this extends
to a short exact sequence of presheaves
\begin{equation}\label{eq:comult}
\begin{pspicture}(0,0)(10,1)
%
\rput(-2.5,0){
\rput(5,0){
\pscircle[fillstyle=solid,fillcolor=white](0,0.5){0.45}
\psbezier[showpoints=false](-0.35,0.74)(-0.05,0.5)(0.2,0.85)(0.25,0.5)
\psbezier[showpoints=false](-0.35,0.26)(-0.05,0.5)(0.2,0.15)(0.25,0.5)
}
\psline{>->}(5.75,0.5)(6.75,0.5)
\rput(6.25,0.7){$\scriptstyle{\Delta}$}
\rput(7.5,0){
\pscircle[fillstyle=solid,fillcolor=white](0,0.5){0.45}
\psbezier[showpoints=false](-0.35,0.74)(-0.15,0.65)(-0.1,0.55)(-0.1,0.5)
\psbezier[showpoints=false](-0.35,0.26)(-0.15,0.35)(-0.1,0.45)(-0.1,0.5)
\pscircle(0.15,0.5){0.15}
}
\psline{->>}(8.25,0.5)(9.25,0.5)
\rput(10,0){
\pscircle[fillstyle=solid,fillcolor=white](0,0.5){0.45}
\psbezier[showpoints=false](-0.35,0.74)(-0.05,0.5)(0.2,0.85)(0.25,0.5)
\psbezier[showpoints=false](-0.35,0.26)(-0.05,0.5)(0.2,0.15)(0.25,0.5)
}
}
\end{pspicture}
\end{equation}
The composition $ \epsilon \circ \Delta $ is the
identity map.

Occasionally the link
diagram $D$ will be too large for the circle notation above
(e.g: in \S\ref{section:applications:example}), and so we will
just write $D$, or a shaded version.
For example if $D_1,D_2$ are unoriented link diagrams then (\ref{eq:unit})
extends to
$$
\begin{pspicture}(0,0)(14.5,1)
%
\rput(0,-0.1){
\rput(2.5,0){
\rput(0,0.2){
\rput(0,0){
\psframe(0,0)(0.8,0.8)
\rput(0.4,0.4){$D_1$}
}
\rput(1.1,0.4){
\rput(0,0){
\psline(-0.3,0.24)(0.3,0.24)
}
\rput(0,0){
\psline(-0.3,-0.24)(0.3,-0.24)
}
}
\rput(1.4,0){
\psframe(0,0)(0.8,0.8)
\rput(0.425,0.4){$D_2$}
}}
}
%
%
\rput(5,0.2){
\psline{>->}(0,0.4)(0.9,0.4)
\rput(0.45,0.6){${\scriptstyle\iota}$}
}
%
\rput(6.15,0){
\rput(0,0.2){
\rput(0,0){
\psframe(0,0)(0.8,0.8)
\rput(0.4,0.4){$D_1$}
}
\rput(1.1,0.4){
\rput(0.05,0){
\psbezier[showpoints=false](-0.35,0.24)(-0.2,0.15)(-0.2,0.05)(-0.2,0)
\psbezier[showpoints=false](-0.35,-0.24)(-0.2,-0.15)(-0.2,-0.05)(-0.2,0)
}
\rput(-0.05,0){
\psbezier[showpoints=false](0.35,0.24)(0.2,0.15)(0.2,0.05)(0.2,0)
\psbezier[showpoints=false](0.35,-0.24)(0.2,-0.15)(0.2,-0.05)(0.2,0)
}
}
\rput(1.4,0){
\psframe(0,0)(0.8,0.8)
\rput(0.4,0.4){$D_2$}
}}
}
%
%
\rput(8.6,0.2){
\psline{->>}(0,0.4)(0.9,0.4)
\rput(0.45,0.6){${\scriptstyle\varepsilon}$}
}
%
\rput(9.8,0){
\rput(0,0.2){
\rput(0,0){
\psframe(0,0)(0.8,0.8)
\rput(0.4,0.4){$D_1$}
}
\rput(1.1,0.4){
\rput(0,0){
\psline(-0.3,0.24)(0.3,0.24)
}
\rput(0,0){
\psline(-0.3,-0.24)(0.3,-0.24)
}
}
\rput(1.4,0){
\psframe(0,0)(0.8,0.8)
\rput(0.4,0.4){$D_2$}
}}
}
%
}
\end{pspicture}
$$
and there are similar sequences for $m$ and $\Delta$.

\subsection{The skein relation}
\label{section:applications:skein}

By choosing a crossing there is an evident smoothing change morphism
of presheaves:
$$
\begin{pspicture}(0,0)(10,1)
%
\rput(-2.5,0){
\rput(6.25,0){
\pscircle[fillstyle=solid,fillcolor=white](0,0.5){0.45}
\pscircle[fillstyle=solid,fillcolor=white](0.260,0.840){0.090}
\psbezier[showpoints=false](-0.25,0.7)(-0.2,0.65)(-0.15,0.6)(0,0.6)
\psbezier[showpoints=false](0.25,0.7)(0.2,0.65)(0.15,0.6)(0,0.6)
\psbezier[showpoints=false](-0.25,0.3)(-0.2,0.35)(-0.15,0.4)(0,0.4)
\psbezier[showpoints=false](0.25,0.3)(0.2,0.35)(0.15,0.4)(0,0.4)
}
\psline{->}(7,0.5)(8,0.5)
\rput(7.5,0.7){$\scriptstyle{\varphi}$}
\rput(8.75,0){
\pscircle[fillstyle=solid,fillcolor=white](0,0.5){0.45}
\pscircle[fillstyle=solid,fillcolor=white](0.260,0.840){0.090}
\psbezier[showpoints=false](-0.2,0.75)(-0.15,0.7)(-0.1,0.65)(-0.1,0.5)
\psbezier[showpoints=false](-0.2,0.25)(-0.15,0.3)(-0.1,0.35)(-0.1,0.5)
\psbezier[showpoints=false](0.2,0.75)(0.15,0.7)(0.1,0.65)(0.1,0.5)
\psbezier[showpoints=false](0.2,0.25)(0.15,0.3)(0.1,0.35)(0.1,0.5)
}
}
\end{pspicture}
$$
In general this is neither surjective nor injective. There is however
an induced  map of spaces
$$
\begin{pspicture}(0,0)(10,1)
%
\rput(-2.5,0){
\rput(6.25,0){
\rput(-1,0.5){$\holim$}
\pscircle[fillstyle=solid,fillcolor=lightergray](0,0.5){0.45}
\pscircle[fillstyle=solid,fillcolor=white](0.260,0.840){0.090}
\psbezier[showpoints=false](-0.25,0.7)(-0.2,0.65)(-0.15,0.6)(0,0.6)
\psbezier[showpoints=false](0.25,0.7)(0.2,0.65)(0.15,0.6)(0,0.6)
\psbezier[showpoints=false](-0.25,0.3)(-0.2,0.35)(-0.15,0.4)(0,0.4)
\psbezier[showpoints=false](0.25,0.3)(0.2,0.35)(0.15,0.4)(0,0.4)
}
\psline{->}(7,0.5)(8,0.5)
\rput(7.5,0.7){$\scriptstyle{\ov{\varphi}}$}
\rput(9.6,0){
\rput(-1,0.5){$\holim$}
\pscircle[fillstyle=solid,fillcolor=lightergray](0,0.5){0.45}
\pscircle[fillstyle=solid,fillcolor=white](0.260,0.840){0.090}
\psbezier[showpoints=false](-0.2,0.75)(-0.15,0.7)(-0.1,0.65)(-0.1,0.5)
\psbezier[showpoints=false](-0.2,0.25)(-0.15,0.3)(-0.1,0.35)(-0.1,0.5)
\psbezier[showpoints=false](0.2,0.75)(0.15,0.7)(0.1,0.65)(0.1,0.5)
\psbezier[showpoints=false](0.2,0.25)(0.15,0.3)(0.1,0.35)(0.1,0.5)
}
}
\end{pspicture}
$$
and we can easily describe its homotopy fibre to give
a homotopy theoretic incarnation of the skein
relation:

\begin{proposition}\label{prop:skein}
\hspace{1em}
$\hofibre(
\holim 
\begin{pspicture}(0,0)(1,0.5)
%
\rput(0.5,-0.4){
\pscircle[fillstyle=solid,fillcolor=lightergray](0,0.5){0.45}
\pscircle[fillstyle=solid,fillcolor=white](0.260,0.840){0.090}
\psbezier[showpoints=false](-0.25,0.7)(-0.2,0.65)(-0.15,0.6)(0,0.6)
\psbezier[showpoints=false](0.25,0.7)(0.2,0.65)(0.15,0.6)(0,0.6)
\psbezier[showpoints=false](-0.25,0.3)(-0.2,0.35)(-0.15,0.4)(0,0.4)
\psbezier[showpoints=false](0.25,0.3)(0.2,0.35)(0.15,0.4)(0,0.4)
}
\end{pspicture}
\rightarrow
\holim 
\begin{pspicture}(0,0)(1,0.5)
\rput(0.5,-0.4){
\pscircle[fillstyle=solid,fillcolor=lightergray](0,0.5){0.45}
\pscircle[fillstyle=solid,fillcolor=white](0.260,0.840){0.090}
\psbezier[showpoints=false](-0.2,0.75)(-0.15,0.7)(-0.1,0.65)(-0.1,0.5)
\psbezier[showpoints=false](-0.2,0.25)(-0.15,0.3)(-0.1,0.35)(-0.1,0.5)
\psbezier[showpoints=false](0.2,0.75)(0.15,0.7)(0.1,0.65)(0.1,0.5)
\psbezier[showpoints=false](0.2,0.25)(0.15,0.3)(0.1,0.35)(0.1,0.5)
}
\end{pspicture}
)\simeq
\holim 
\begin{pspicture}(0,0)(1,0.5)
%
\rput(0.5,-0.4){
\pscircle[fillstyle=solid,fillcolor=lightergray](0,0.5){0.45}
\pscircle[fillstyle=solid,fillcolor=white](0.260,0.840){0.090}
\psline(0.2,0.75)(-0.2,0.25)
\psframe[fillstyle=solid,fillcolor=lightergray,linecolor=lightergray](-0.075,0.425)(0.075,0.575)
\psline(-0.2,0.75)(0.2,0.25)
}
\end{pspicture}
$
\end{proposition}

\begin{proof}
We have
$$
\holim 
\begin{pspicture}(0,0)(1,0.5)
\rput(0.5,-0.4){
\pscircle[fillstyle=solid,fillcolor=lightergray](0,0.5){0.45}
\pscircle[fillstyle=solid,fillcolor=white](0.260,0.840){0.090}
\psline(0.2,0.75)(-0.2,0.25)
\psframe[fillstyle=solid,fillcolor=lightergray,linecolor=lightergray](-0.075,0.425)(0.075,0.575)
\psline(-0.2,0.75)(0.2,0.25)
}
\end{pspicture}
\simeq
\holim\left(
\begin{pspicture}(0,0)(1.2,1.3)
\rput(0.6,-1.1){
\psline(0,2)(0,0.5)
\rput(0,2){
\pscircle[fillstyle=solid,fillcolor=lightergray](0,0){0.45}
\pscircle[fillstyle=solid,fillcolor=white](0.260,0.340){0.090}
\rput(0,-0.5){
\psbezier[showpoints=false](-0.25,0.7)(-0.2,0.65)(-0.15,0.6)(0,0.6)
\psbezier[showpoints=false](0.25,0.7)(0.2,0.65)(0.15,0.6)(0,0.6)
\psbezier[showpoints=false](-0.25,0.3)(-0.2,0.35)(-0.15,0.4)(0,0.4)
\psbezier[showpoints=false](0.25,0.3)(0.2,0.35)(0.15,0.4)(0,0.4)
}
}
\rput(0.2,1.25){${\scriptstyle\varphi}$}
\rput(0,0.5){
\pscircle[fillstyle=solid,fillcolor=lightergray](0,0){0.45}
\rput(0,-0.5){
\psbezier[showpoints=false](-0.2,0.75)(-0.15,0.7)(-0.1,0.65)(-0.1,0.5)
\psbezier[showpoints=false](-0.2,0.25)(-0.15,0.3)(-0.1,0.35)(-0.1,0.5)
\psbezier[showpoints=false](0.2,0.75)(0.15,0.7)(0.1,0.65)(0.1,0.5)
\psbezier[showpoints=false](0.2,0.25)(0.15,0.3)(0.1,0.35)(0.1,0.5)
}}}
\end{pspicture}
\right)
$$
and the result follows immediately from Proposition \ref{prop:hohoho}.
\end{proof}

The associated long exact homotopy sequence can  be
identified with the usual 
long exact skein sequence in Khovanov homology (see \cite{Viro02} and \cite{Turner06}).

\subsection{Reidemeister invariance.}
\label{section:applications:reidemeister}

We now give a homotopy theoretic proof of the invariance of
Khovanov homology by
Reidemeister moves (Figure \ref{fig:reidemeister}). The original
proofs can be found in \cite{Khovanov00} (see also \cite{Bar-Natan02})
and more a geometrical argument can be found in \cite{Bar-Natan05}. We
recall that the rank of the underlying Boolean lattice is the number
of crossings in the given diagram, thus moves $(I\pm)$ and $(II)$ alter the underlying
Boolean lattice.

Recalling from the remarks
at the end of Section \ref{sec:spaces} that a negative degree shift in
Khovanov homology is equivalent to taking the loop space, we see that
we must prove
$(I+). \YYn D  \simeq \YYn D^\prime$,
$(I-). \YYn D  \simeq \Omega \YYn {D^\prime}$,
$(II). \YYn D  \simeq \Omega \YYn {D^\prime}$ and
$(III). \YYn D  \simeq \YYn D^\prime$.

\subsubsection{Reidemeister moves $(I\pm)$.}

Let $D$ and $D^\prime$ be two (unoriented)  link diagrams locally
described as in $(I-)$ above. The short exact sequence (\ref{eq:comult}) and
Corollary \ref{cor:holimhofibre}(ii) give
\begin{align*}
\YYn{D} =
\holim\left(

=\Omega\YYn{D'}.
\end{align*}
A completely analogous argument, using 
(\ref{eq:mult}) and Corollary
\ref{cor:holimhofibre}(i), gives
Reidemeister $(I+)$.
  \caption{Reidemeister moves}
  \label{fig:reidemeister}
\end{figure}

\subsubsection{Reidemeister move $(II)$.}

Let $D$ and $D^\prime$ be two link diagrams locally
described as in $(II)$ and let $F_{KH}$ be the
Khovanov presheaf for $D$. 
There is a short exact sequence
$H\rightarrowtail F_{KH}\twoheadrightarrow G$: 

%
  

We leave it to the reader to check that $G$ and $H$ are indeed
presheaves. 
All missing horizontal maps are either the identity or zero (it should be
clear which is which), $\iota$ and $\epsilon$ are the unit and counit, 
and we are using the short exact sequence (\ref{eq:unit}).
To check that we have morphisms of presheaves we
need to show that $\varepsilon d_1=1$ and $d_3\iota=1$. The former
follows from $\varepsilon\Delta=1$ and the latter from $m\iota=1$. 

We have
$$
%
$$
with the homotopy equivalences by 
Propositions \ref{prop:holimhofibre}(i), \ref{prop:holim4} and
Lemma \ref{lem:holim2} respectively.
The map induced by $F_{KH}\twoheadrightarrow G$ is thus a 
homotopy equivalence by the presheaf computational tool,
and so $\YYn {D}$ equals
$$
\holim\left(
%
\right)
$$
with the last homotopy equivalent to 
$\Omega\holim
%
=\Omega\YYn{D}'$ by Proposition
\ref{prop:holim4}(ii) 
and Lemma \ref{lem:holim1}.

\subsubsection{Reidemeister move $(III)$.}

Let $D$ and $D^\prime$ be two (unoriented)  link diagrams locally
described as in $(III)$ above and let $F_{KH}$ be the
Khovanov presheaf for $D$. We start with a
short exact sequence $G_0\rightarrowtail F_{KH}\stackrel{\tau}{\twoheadrightarrow} G$
defined by:
$$
%
  
$$
One can check that $G$ is indeed a presheaf and that $\tau$ is surjective.
All missing horizontal maps are
either the identity or zero.
Proposition
\ref{prop:holim4}(ii) gives
$
\holim
%
$
by the presheaf tool
(abbreviating $F_{KH}$ to $F$).

We now define another presheaf $H$  and a presheaf 
map $\ss:G\rightarrow H$ 
where all missing maps are either the identity or
zero as before.  For this we note that in $G$ we have $\ve
d_5=\ve d_8=\Delta\ve=1$ and  $d_1=d_3$:
$$
%
$$
Now, $\ss$ is neither injective or surjective so we turn to the
homotopy tool.
Using the
contractibility of many 
of the hofibres we have from Lemma \ref{lem:holim1A} and Corollary \ref{cor:holim4B} that
$$
%
$$
with the second to last homotopy equivalence from Corollary \ref{cor:holimhof}, 
Lemma \ref{lem:holim2}  and Proposition \ref{prop:holim4}. 
Thus by Proposition \ref{prop:hofibre}
$$
%
$$
It follows that  $\ov{\ss}$
is a homotopy equivalence from which we obtain
$$
\YYn{D}=\holim
%
$$
Now repeat the entire process starting with $F_{KH}'$, the Khovanov
presheaf for ${D}'$: 
a short exact sequence $G_0'\rightarrowtail
F_{KH}'\twoheadrightarrow G'$ and a morphism $\ss':G'\ra H'$ can be
defined in a completely analogous way and the homotopy tool invoked to give
$$
\YYn{D^\prime}=\holim
%
$$
Comparing $H$ and $H'$, it turns out that the vertex groups are visibly identical, as are the
edge maps except for an occurrence of $d_9d_1$ in $H$ and $d_7d_2$ in $H'$. However
the front top face of $F$ shows that $d_9d_1=d_7d_2$  and so $H$ and
$H^\prime$ are in fact \emph{identical\/} diagrams. Thus, $\YYn{D}\simeq
\YYn{D^\prime}$, completing the proof of Reidemeister $(III)$.

\subsection{An example.}
\label{section:applications:example}

We take the technology for a test drive by showing that 
$$
%
$$
where $D_1,D_2$ are (unoriented) link diagrams and we are simplifying
our pictograms as in the remarks immediately before
\S\ref{section:applications:skein}. The conclusions for Khovanov
homology are at the end of this calculation, and although they could be acheived at the
level of chain complexes, our purpose here is to illustrate our
machinery in action.

The Skein relation (Proposition \ref{prop:skein}) gives
$$
%
$$
which in turn is homotopy equivalent to
$$
%
$$
Consider the following two diagrams over $\BQ\times\BD$ 
(with $\BD=\blob\longrightarrow\blob\longleftarrow\blob$)
and the morphism between them:
$$
%
$
is the trivial diagram of spaces.
We have the short exact sequence $H\rightarrowtail
F\twoheadrightarrow G$ of presheaves:
$$
%
$
the trivial presheaf, 
and $H'\rightarrowtail
F'\twoheadrightarrow G'$:
$$
%
$$
As $\scrG,\scrG'\simeq\star$ by Lemma \ref{lem:holim2}, 
the presheaf computational tool gives
$\scrH\simeq\scrF$, $\scrH'\simeq\scrF'$. 
Proposition \ref{prop:holimhomotopy} applied to the morphism of diagrams over
$\BQ\times\BD$ thus gives that
$(\dag)$ is homotopy equivalent to
$$
%
$$
where $\stackrel{\star}{\longrightarrow}$ indicates that the induced
map of the $\holim$'s factors through a point.
The result then follows using
$\hofibre(X\stackrel{\star}{\longrightarrow}Y)\simeq X\times\Omega
Y$, Lemma \ref{lem:holim1}  and Corollary \ref{cor:holim4B}.

To convert to a statement about Khovanov homology let $D_1,D_2$ now be
oriented, with
$$
%
$$
Then by the discussion at the end of \S\ref{sec:spaces} 
$$
\YY (D_1\#\# D_2)\simeq \YY(D_1\# D_2)\times\Omega^2\YY(D_1\# D_2)
$$
so that
$$
\XX (D_1\#\# D_2)\simeq \XX(D_1\# D_2)\times\Omega^2\XX(D_1\# D_2)
$$
as the two additional crossings in $D_1\#\# D_2$ are both
positive. Thus
\begin{align*}
  \label{eq:11}
KH^i(D_1\#\# D_2)&\cong\pi_{-i}(\XX(D_1\# D_2)\times\Omega^2\XX(D_1\#
D_2))\\
&\cong\pi_{-i}\XX(D_1\# D_2)\oplus\pi_{-i+2}\XX(D_1\#D_2)\\
&\cong KH^i(D_1\# D_2)\oplus KH^{i-2}(D_1\# D_2)
\end{align*}

For example
$$
%

%
%
%
\bibliography{AGT_style_version}{}
\bibliographystyle{gtart}

\end{document}